\documentclass[aop,MSNbibl,citesort,seceqn,dvips]{arximspdf}
\usepackage{mathbh}

\doi{10.1214/11-AOP670}
\volume{40}
\issue{5}
\pubyear{2012}
\firstpage{2069}
\lastpage{2105}

\makeatletter

\newcommand{\st}{>}
\newcommand{\newW}{\overline{W}}

\newcommand{\N}{\mathbb{N}}
\newcommand{\Z}{\mathbb{Z}}
\newcommand{\R}{\mathbb{R}}
\newcommand{\B}{\mathfrak{B}}
\newcommand{\V}{\mathbb{V}}

\newcommand{\Prob}{\mathbb{P}}
\newcommand{\E}{\mathbb{E}}
\newcommand{\G}{\mathbb{G}}
\newcommand{\A}{\mathcal{A}}
\newcommand{\1}{\mathbh{1}}
\newcommand{\bT}{\mathbf{T}}
\newcommand{\Kl}{K_{\mathrm{l}}}
\newcommand{\Ku}{K_{\mathrm{u}}}
\newcommand{\eqref}[1]{(\ref{#1})}

\newcommand{\SolM}{\mathcal{S}(\mathcal{M})}
\newcommand{\SolL}{\mathcal{S}(\mathcal{L})}
\newcommand{\SolU}{\mathcal{S}(\mathcal{U})}

\newtheorem{Theorem}{Theorem}[section]
\newtheorem{Prop}[Theorem]{Proposition}
\newtheorem{Cor}[Theorem]{Corollary}
\newtheorem{Lemma}[Theorem]{Lemma}

\newproclaim{Def}[Theorem]{Definition}

\makeatother

\begin{document}
\begin{frontmatter}

\title{The functional equation of the smoothing~transform}
\runtitle{The functional equation of the smoothing transform}

\begin{aug}
\author[A]{\fnms{Gerold} \snm{Alsmeyer}\ead[label=e1]{gerolda@math.uni-muenster.de}},
\author[B]{\fnms{J. D.} \snm{Biggins}\ead[label=e2]{j.biggins@sheffield.ac.uk}}
and
\author[C]{\fnms{Matthias} \snm{Meiners}\corref{}\ead[label=e3]{matthias.meiners@math.uu.se}\thanksref{t1}}
\thankstext{t1}{Supported in part by DFG-Grant Me 3625/1-1.}
\runauthor{G. Alsmeyer, J. D. Biggins and M. Meiners}
\affiliation{University of M\"unster, University of Sheffield and
Uppsala University}
\address[A]{G. Alsmeyer \\
Institut f\"ur Mathematische Statistik \\
Universit\"at M\"unster \\
Einsteinstra\ss e 62\\
DE-48149 M\"unster\\
Germany\\
\printead{e1}}

\address[B]{J. D. Biggins \\
School of Mathematics and Statistics \\
University of Sheffield \\
Sheffield S3 7RH\\
United Kingdom \\
\printead{e2}}

\address[C]{M. Meiners \\
Matematiska institutionen \\
Uppsala universitet \\
Box 480\\
751 06 Uppsala\\
Sweden\\
\printead{e3}}
\end{aug}

\received{\smonth{3} \syear{2010}}
\revised{\smonth{1} \syear{2011}}

%
\begin{abstract}
Given a sequence $T = (T_i)_{i \geq1}$ of nonnegative random
variables, a~function $f$ on the positive halfline can be transformed
to $\mathbb{E} \prod_{i \geq1} f(tT_i)$. We study the fixed points
of this
transform within the class of decreasing functions. By exploiting the
intimate relationship with general branching processes, a full
description of the set of solutions is established without the moment
conditions that figure in earlier studies. Since the class of functions
under consideration contains all Laplace transforms of probability
distributions on $[0,\infty)$, the results provide the full description
of the set of solutions to the fixed-point equation of the smoothing
transform, $X\stackrel{d}{=}\sum_{i \geq1} T_i X_i$, where
$\stackrel{d}{=}$ denotes equality of the corresponding laws, and $X_1,
X_2, \ldots$ is a sequence of i.i.d.\ copies of $X$ independent of $T$.
Further, since left-continuous survival functions are covered as well,
the results also apply to the fixed-point equation
$X\stackrel{d}{=}\inf\{ X_i/T_i\dvtx i \geq1, T_i > 0\}$. Moreover, we
investigate the phenomenon of endogeny in the context of the smoothing
transform and, thereby, solve an open problem posed by Aldous and
Bandyopadhyay.
\end{abstract}

%


\end{frontmatter}

\section{Introduction} \label{sec:Intro}

Let $T := (T_i)_{i \geq1}$ be a sequence of nonnegative random
variables, and consider the mapping $f \mapsto\E\prod_{i \geq1} f(t
T_i)$ for suitable functions $f\dvtx\R\to\R$. Then it is natural to call
$f$ a fixed point of this transformation if
\begin{equation} \label{eq:FE1}
f(t) = \E\prod_{i \geq1} f(t T_i).
\end{equation}
The main objective here is to identify all fixed points within certain
classes of functions, which becomes an increasingly challenging task as
the available class gets bigger.
There is a substantial literature, \cite
{Big1977,DL1983,Liu1998,Kyp1998,Iks2004,BK2005}, and relatively
complete results, improved here, when $f$ must be the Laplace transform
of a nonnegative random variable. Much less was known up to now \cite
{JR2004,AR2008,AM2009} when $f$ is from the larger class of survival
functions of nonnegative random variables (or simply monotone
decreasing functions with range $[0,1]$). Solving the problem in this
case is one of the main achievements of this paper. In fact the ideas
also allow the available class to include suitable nonmonotonic
functions, as will be indicated in the final section.

When $f$ is the Laplace transform of a nonnegative variable $X$,
equation~\eqref{eq:FE1} can be rewritten in terms of random variables as
\begin{equation} \label{eq:SumFP}
X \stackrel{d}{=} \sum_{i \geq1} T_i X_i,
\end{equation}
where $X_1, X_2, \ldots$ are i.i.d.\ copies of $X$ independent of $T$, and
$\stackrel{d}{=}$ means equality in distribution. An $X$, or its
distribution, satisfying this is often called a~fixed point of the
smoothing transform (going back to Durrett and Liggett~\cite{DL1983}).
If instead $f$ is the (left-continuous) survival function of a
nonnegative variable $X$, equation~\eqref{eq:FE1} can be rewritten as
\begin{equation} \label{eq:MinFP}
X \stackrel{d}{=} \inf \biggl\{ \frac{X_i}{T_i}\dvtx i \geq1, T_i > 0 \biggr\},
\end{equation}
where the infimum over the empty set is defined to be $\infty$.
A solution $X$, and the associated survival function $P(X \geq t)$, is
called nondegenerate if \mbox{$\Prob(X \in(0,\infty)) > 0$}.
The inversion $x \mapsto x^{-1}$ turns this ``inf-type'' equation into a
``sup-type'' one, so the theory will cover these too.
Both~\eqref{eq:SumFP} and~\eqref{eq:MinFP} are examples of \emph
{stochastic fixed-point equations} (also called \emph{recursive
distributional equations} in~\cite{AB2005}). Thus, for these two
cases, characterizing fixed points for equation~\eqref{eq:FE1} in the
appropriate class corresponds to identifying the $X$ which can arise in
these stochastic fixed-point equations. In considering \eqref
{eq:SumFP}, the relevant class of functions (Laplace transforms) is
quite restricted, and so the problem is correspondingly easier. It
turns out that solutions to
\eqref{eq:SumFP} are intimately related to solutions to \eqref
{eq:MinFP}, which allows the characterization of the latter using
results for the former.

There is considerable interest in, and literature on, stochastic
fixed-point equations like~\eqref{eq:SumFP} and~\eqref{eq:MinFP}.
They occur in various areas of applied probability:
probabilistic combinatorial optimization~\cite{AS2004}, stochastic
geometry~\cite{PW2006}, the analysis of recursive algorithms and data
structures~\cite{Roe1991,GruRoe1996,Dev2001,RoeRue2001,NR2005} and
also in connection with branching particle systems~\cite{BK1997,HS2009}.
Inhomogeneous versions of~\eqref{eq:SumFP} and the sup-type version of
\eqref{eq:MinFP} arise in the average-case and worst-case analysis of
divide-and-conquer algorithms, and
R{\"u}schendorf~\cite{Rue2006}, Theorem 3.1 and Theorem 4.2, showed, in
a more restricted setting, that the solutions to the inhomogeneous
versions are in one-to-one\vadjust{\goodbreak} correspondence with the solutions of their
homogeneous counterparts.
In theoretical probability, they are of relevance in connection with
the central limit problem~\cite{Bre1992} and in extreme value theory
\cite{Res1987}, where they can be interpreted as generalizations of
the distributional equations of stability and min-stability,
respectively. For further information we refer to the survey by Aldous
and Bandyopadhyay~\cite{AB2005}.

Without loss of generality, suppose
that the number $N = \sum_{i \geq1} \1_{\{T_i > 0\}}$ of positive
terms satisfies
$N = \sup\{i \geq1\dvtx T_i > 0\}$,
and define the function
\[
m\dvtx[0,\infty) \to[0,\infty],\qquad
\theta \mapsto \E\sum_{i=1}^N T_i^{\theta}.
\]
Its canonical domain, $\{m < \infty\}$, is an interval $\subseteq
[0,\infty)$, for $m$ may be viewed as the Laplace transform of the
intensity measure of the point process $\mathcal{Z} := \sum_{i=1}^N
\delta_{S(i)}$.
Here $S(i) := -\log T_i$, $i \in\N$ (and $S(i) := \infty$ if $T_i = 0$).
The following assumptions will be in force
throughout.
{\renewcommand{\theequation}{A\arabic{equation}}
\setcounter{equation}{0}
\begin{eqnarray} \label{eq:A1}
\Prob (T \in\{0,1\}^{\N} ) &<& 1.
\\[-1pt]
\label{eq:A2}
\E N &>& 1.
\end{eqnarray}\vspace*{-24pt}
\begin{equation}
\label{eq:A3}
\mbox{There is an } \alpha> 0 \mbox{ such that } 1 = m(\alpha) <
m(\beta) \mbox{ for all } \beta\in[0,\alpha).
\end{equation}}
\vspace*{-\baselineskip}
\setcounter{equation}{3}

\noindent This number $\alpha$ is called the \emph{characteristic exponent} (\textit{of}
$T$). Previous~\cite{DL1983,Liu1998,AR2006,AR2008} and recent \cite
{AM2010a} studies show that a satisfactory characterization will
typically entail the existence of some $\alpha> 0$ such that $m(\alpha
)=1$, as in~\eqref{eq:A3},
though~\cite{JR2004} and~\cite{AM2009} provide a study of a case
where this fails. The discussions in~\cite{Liu1998} for equation\
\eqref{eq:SumFP} and~\cite{JR2004,AM2009} for equation\ \eqref
{eq:MinFP} imply that only simple cases are ruled out by~\eqref{eq:A1}
and~\eqref{eq:A2}. Let $r>1$ be the smallest number such that the
strictly positive elements of $T$ are concentrated almost surely on
$r^{\Z}$, and let $r$=1 otherwise, that is, when the smallest closed
(in~$\R^+$) multiplicative group containing the strictly positive
elements of $T$ is~$\R^+$. The former is called the \emph
{$r$-geometric} (or \emph{lattice}) case, the latter the \emph
{nongeometric} (or \emph{continuous}) case.
There are more technicalities to deal with before the main results can
be stated but a special case is given now as illustration.\vspace*{-2pt}

\begin{Theorem} \label{Thm:Min_solutions_first}
Suppose that~\eqref{eq:A1}--\eqref{eq:A3} hold true, that $\Prob
(N<\infty)=1$, $m(\theta)<\infty$ for some $\theta\in[0,\alpha)$
and that $T$ is nongeometric. Then there exists a
nonnegative and finite random variable $W$ satisfying $\Prob(W>0)>0$ and
\begin{equation} \label{eq:sm-fe}
W \stackrel{d}{=} \sum_{i \geq1} T_i^{\alpha} W_i,
\end{equation}
(where $W_1, W_2,\ldots$ are i.i.d.\ copies of $W$ independent of $T$)
such that nondegenerate survival functions that are solutions to~\eqref{eq:FE1}
are given by the family, parametrized by $h \in\R^+$,
\begin{equation} \label{eq:W-mix}
f(t) = \E\exp(-W h t^\alpha).\vspace*{-2pt}
\end{equation}
\end{Theorem}

Note that~\eqref{eq:sm-fe} is just~\eqref{eq:SumFP} with $T$ replaced
by $T^{(\alpha)}:=(T_1^\alpha, T_2^\alpha, \ldots)$. It is already
known, under mild conditions that are relaxed a little\vadjust{\goodbreak} in Theorem~\ref
{Thm:essential_uniqueness_and_reg_var} here, that solutions to~\eqref{eq:sm-fe}
of the form described in Theorem~\ref{Thm:Min_solutions_first} are
unique up to a scale factor. Therefore, in~\eqref{eq:W-mix},
the same family will result, whichever solution to~\eqref{eq:sm-fe} is
selected. Form~\eqref{eq:W-mix} is a mixture (with mixing variable~$W$)
of Weibull survival functions. This form is not surprising in the
light of results for deterministic $T$ described in~\cite{AR2008}
and the corresponding results for~\eqref{eq:SumFP} going back to
Durrett and Liggett~\cite{DL1983}. In the latter case, $f$ has to be a
Laplace transform and~\eqref{eq:W-mix} expresses it as a $W$-mixture
of positive $\alpha$-stable transforms (necessitating also that
$\alpha\leq1$).

It is natural to deploy iteration to study a functional equation. A key
aspect of the approach here is to remove the expectation on the right
of~\eqref{eq:FE1} and then iterate. Suitably formulated, this
iteration derives naturally from a branching process based on $T$.
Solutions to~\eqref{eq:FE1} correspond to certain (multiplicative)
martingales. Studying these, and their limits, delivers information on
the form of the solutions. This basic idea goes back at least to Neveu
\cite{Nev1988} and is used more recently in~\cite{BK1997,BK2005} and
\cite{AM2009}.
This technique is a~kind of disintegration of~\eqref{eq:FE1}, since it
considers the stochastic processes obtained by removing the expectation
in it and its iterates. For fixed $t$, under the iteration of the
disintegration, the conditions imply that the arguments of the function
$f$ on the right of the equation become small. Hence, the properties of
the whole function will be implicit in its behavior for small arguments.

Further, our approach brings to the forefront a fundamental property of
solutions to~\eqref{eq:sm-fe}: \emph{endogeny}.
Heuristically speaking, a solution $W$ to~\eqref{eq:sm-fe} is
endogenous if
$W$ can be constructed from the branching process mentioned above
without additional randomization.
In their survey paper, Aldous and Bandyopadhyay posed the open problem
of studying the endogeny property in the context of the smoothing
transform~\cite{AB2005}, Open Problem 18. In Section~\ref{sec:end_FP}
(see Theorem~\ref{Thm:endogeny}), we give the solution to this problem
under mild conditions.

\section{Main results} \label{sec:Main_results}

We continue with further assumptions on $T$, namely,
{\renewcommand{\theequation}{A4\alph{equation}}
\setcounter{equation}{0}
\begin{eqnarray}  \label{eq:A4a}
\qquad &\displaystyle \E\sum_{i \geq1} T_i^{\alpha} \log T_i \in (-\infty,0)
\quad \mbox{and}\quad
\E \biggl(\sum_{i \geq1} T_i^{\alpha} \biggr) \log^+ \biggl(\sum_{i \geq1}
T_i^{\alpha} \biggr) < \infty.
\\
\label{eq:A4b}
\qquad &\mbox{There exists some } \theta\in[0,\alpha) \mbox{ satisfying }
m(\theta) < \infty.
\end{eqnarray}}
\vspace*{-\baselineskip}

\noindent In order to prove our main results, we need at least one of the
assumptions,~\eqref{eq:A4a},~\eqref{eq:A4b}, to be true; in other
words, we need the following assumption:
{\renewcommand{\theequation}{A\arabic{equation}}
\setcounter{equation}{3}
\begin{equation} \label{eq:A4}
\eqref{eq:A4a} \mbox{ or }~\eqref{eq:A4b} \mbox{ holds}.
\end{equation}
}
\vspace*{-\baselineskip}

\noindent It is worth mentioning that~\eqref{eq:A4} is fairly weak compared to
the assumptions in earlier works on fixed points of the smoothing
transform, that is, on solutions to~\eqref{eq:SumFP}. For ease of
reference to earlier results, when~\eqref{eq:A3} holds let
\[
m'(\alpha)=\E\sum_{i \geq1} T_i^{\alpha} \log T_i,
\]
even when $m$ is finite only at $\alpha$; whenever we refer to
$m'(\alpha)$ we will be assuming the expectation exists, which it
certainly does when~\eqref{eq:A4} holds.

We impose one further assumption. To state it, call a random variable
$W$ nonnull if $\Prob(W \not= 0) > 0$, and assume that~\eqref{eq:A3} holds.
{\renewcommand{\theequation}{A\arabic{equation}}
\setcounter{equation}{4}
\begin{equation} \label{eq:E}
\qquad\   \begin{tabular}{@{}p{325pt}@{}}
There is a finite, nonnull, nonnegative random variable
$W$, with Laplace transform $\varphi$, satisfying \eqref
{eq:sm-fe}.
\end{tabular}\hspace*{-10pt}
\end{equation}}
\vspace*{-\baselineskip}
\setcounter{equation}{0}

\noindent The next result indicates that this assumption is known to hold widely.
It follows directly from Theorem 1.1 in~\cite{Liu1998} when $\Prob
(N<\infty)=1$ and, as is explained further in
Section~\ref{sec:end_FP}, from~\cite{Lyo1997} when~\eqref{eq:A4a} holds.

\begin{Prop}\label{existence}
If~\eqref{eq:A1}--\eqref{eq:A3} and, furthermore, either \eqref
{eq:A4a} or $\Prob(N<\infty)$ $=1$ hold true, then so does~\eqref{eq:E}.
\end{Prop}

The $r$-geometric case involves some complications that require
additional notation.
A function $h$ is multiplicatively $r$-periodic if $h(x)=h(rx)$ for all $x$.
Given $r>1$, let $\mathfrak{H}_r$ be the set of multiplicatively
$r$-periodic functions $h\dvtx\R^{+}\rightarrow\R^{+}$
such that $t \mapsto h(t) t^{\alpha}$ is nondecreasing [where $\alpha
$ comes from~\eqref{eq:A3}].
To deal with all cases together, let $\mathfrak{H}_1$ be the positive
constant functions in the nongeometric case (when $r=1$).
In the corresponding result for~\eqref{eq:SumFP}, stated here as a
corollary, it is further assumed that $\alpha\in(0,1]$. Then, let
$\mathfrak{P}_r$ be the set of multiplicatively $r$-periodic functions
$h\dvtx\R^{+}\rightarrow\R^{+}$ such that $h(t)t^{\alpha}$ has a
completely monotone derivative, and let $\mathfrak{P}_1$ be the
positive constant functions in the nongeometric case---these
functions were introduced in~\cite{DL1983}.
When $\alpha=1$, the requirements force $h$ to be constant, so that in
this case $\mathfrak{P}_r \equiv\mathfrak{P}_1$.

Henceforth, let $\SolM$ be the set of solutions to the functional
equation~\eqref{eq:FE1} within the class
\begin{eqnarray*}
\mathcal{M}
& = & \{f\dvtx[0,\infty) \to[0,1]\dvtx \mbox{$f$ is decreasing with } \\
& &
\ f(0) = f(0+) = 1 \mbox{ and } f(t) \in(0,1)  \mbox{ some } t > 0\}.
\end{eqnarray*}
Here, $f(0+)$ denotes the right limit of $f$ at $0$. Now we are ready
to state our first main result. This result will be derived from
Theorem~\ref{Thm:Disintegration}, which is more fundamental but needs
more background material to state.

\begin{Theorem} \label{Thm:Min_solutions}
Suppose that~\eqref{eq:A1}--\eqref{eq:E} hold.
Then $\SolM$ is given by the family, parametrized by $h \in\mathfrak{H}_r$,
\begin{equation} \label{eq:F_inf=W_Lambda_alpha}
f(t) = \E\exp(-W h(t) t^\alpha) = \varphi(h(t) t^\alpha) \qquad (t \geq0).
\end{equation}
\end{Theorem}

Let $\SolL$ be the set of solutions to~\eqref{eq:FE1} within the
class $\mathcal{L}$ of Laplace transforms of probability distributions
$\not= \delta_0$ on $[0,\infty)$.

\begin{Cor} \label{Cor:Sum_solutions}
Suppose that conditions~\eqref{eq:A1}--\eqref{eq:E} hold and that
\mbox{$\alpha\leq1$}.
Then $\SolL$ is given by the family in \eqref
{eq:F_inf=W_Lambda_alpha} when parametrized by $h \in\mathfrak{P}_r$.
\end{Cor}

\begin{Cor} \label{Cor:inf-solutions}
Suppose that~\eqref{eq:A1}--\eqref{eq:E} hold.
Then the set of survival functions of nondegenerate solutions to \eqref
{eq:MinFP} is given by the family~\eqref{eq:F_inf=W_Lambda_alpha}
parametrized by the left-continuous $h \in\mathfrak{H}_r$.
\end{Cor}

\section{Further results and discussion} \label{sec:Further_results}

From the formulations of Theorem~\ref{Thm:Min_solutions} and Corollary
\ref{Cor:Sum_solutions} it is obvious that solutions to \eqref
{eq:sm-fe} play a critical role since they appear as mixing
distributions in all other cases. We need information on these fixed
points at an early stage of our analysis. Hence, we continue with the
following results.

\begin{Theorem} \label{Thm:essential_uniqueness_and_reg_var}
$\!\!\!$Assume~\eqref{eq:A1}--\eqref{eq:E} hold.
Let $D(t) := t^{-1}(1-\varphi(t))$. Then~$D$ is slowly varying at $0$,
and $\varphi$ is unique up to a positive scaling factor in its argument.
\end{Theorem}

This result can be concluded from existing literature if we
additionally assume
\eqref{eq:A4b} and $P(N<\infty)=1$. In this case, the claimed
uniqueness of $\varphi$ follows from~\cite{BK2005}, Theorem 3. The
regular variation of $1-\varphi$ follows from~\cite{BK1997}, Theorem
1.4, in the case when $m'(\alpha) < 0$ and from~\cite{Kyp1998}, Theorem
1, if $m'(\alpha) = 0$.
We prove the result as stated here in Section~\ref{sec:unique}; there
we will see that for uniqueness up to scaling~\eqref{eq:A4a} can be
replaced by $\E W < \infty$, which is in fact weaker.

Understanding the behavior of solutions to~\eqref{eq:FE1} near zero is
an essential step in proving Theorem~\ref{Thm:Min_solutions},
but it is also interesting in its own right because it allows the
derivation of moment results for the corresponding distributions by
classical Tauberian theorems in the case of~\eqref{eq:SumFP} and by
elementary calculations for the sup-type analog of~\eqref{eq:MinFP}.

For a solution $f$, the near-zero behavior is best considered in terms
of $D_{\alpha}$ defined by
\[
D_{\alpha}(t) = \frac{1-f(t)}{t^{\alpha}},\qquad  t > 0.
\]
When $\alpha=1$ and $f$ is a Laplace transform, the convexity of $f$
forces $D_1(t) = (1-f(t))/t$, to be decreasing in $t$, and then
$D_1(0+)$ is finite exactly when the corresponding random variable has
a finite mean.

\begin{Theorem} \label{Thm:Slow_variation}
Suppose that~\eqref{eq:A1}--\eqref{eq:E} hold true, that $T$ is
nongeometric and $f \in\SolM$. Then $D_{\alpha}(t)$ is slowly
varying as $t \downarrow0$.
\end{Theorem}

A corresponding result for the geometric case is stated next, although
it is tangential to the development of the results.\vadjust{\goodbreak}
\begin{Theorem} \label{Thm:slow_variation_geometric}
Suppose that~\eqref{eq:A1}--\eqref{eq:E} hold true, that $T$ is
$r$-geometric and $f \in\SolM$. Then there exists a function $h \in
\mathfrak{H}_r$ such that
$D_{\alpha}(t)/h(t)$ is slowly varying as $t \downarrow0$.
\end{Theorem}

It is often possible to say more about the form of the slowly varying
functions. We omit the details but give an indication of the results.
When $\alpha\ne1$, Theorem~\ref{Thm:Min_solutions} gives that any
solution $f\in\SolM$ is of the form $f(t)=\varphi(h(t)t^{\alpha})$
for some $h\in\mathfrak{H}_{r}$, where $\varphi$ comes from \eqref
{eq:E}. Then $t^{-1}(1-\varphi(t))$, which is $D(t)$,
is slowly varying in very specific ways under fairly mild moment
conditions. Roughly speaking, if $m'(1)<0$, then $D(t)$ usually
converges to a finite constant while, if $m'(1)=0$, $D(t)$ usually
looks like $-\log t$. See~\cite{Big1977}, Lemmas~2 and~4, \cite
{Lyo1997} and~\cite{CR2003} for information on the first case and
\cite{BK2005}, Theorems~4 and~5, for information on the second.
It is easy to translate such results on~$\varphi$ to corresponding
results on asymptotic behavior of $f \in\SolM$ at $0$.

We finish this section with a brief summary of previous results on
fixed points of the smoothing transform and the corresponding inf-type
distributional equation. Theorem~\ref{Thm:Min_solutions} has only one
real predecessor, namely, Theorem~4.2 in~\cite{AR2008}, where the
inf-type equation is solved in the case of a deterministic sequence
$(T_1,T_2,\ldots)$. There are (to our knowledge) two further papers
dealing with the inf-type equation:~\cite{JR2004,AM2009}. The first
paper, formulated in terms of the corresponding sup-type equation,
provides a full description of the set of solutions only in very
special cases, while the second one presents an approach that leads to
all solutions only within subclasses of sufficiently regular distributions.
Much more was known about the solutions to~\eqref{eq:FE1} within the
set of Laplace transforms.
For the case $\alpha< 1$, Theorem 1.4 in~\cite{Liu1998} is the result
which gave a full description of the set of solutions under the weakest
conditions so far. However, beyond the conditions required in our
Corollary~\ref{Cor:Sum_solutions}, Liu assumed that $\E N^{1+\delta}
< \infty$ and $\E(\sum_{i \geq1} T_i)^{1+\delta}< \infty$ for
some $\delta> 0$.
Iksanov~\cite{Iks2004}, Proposition 3, gave a description of the set of
solutions under the condition of existence of a so-called elementary
fixed point. However, due to an error in the proof for the case $\alpha
<1$ (personal communication), he later reduced his result to fixed points
within the subclass of Laplace transforms $\phi$ such that $1-\phi$
is regularly varying at the origin (see~\cite{Iks2007}).
In the case $\alpha= 1$, more was already known.
Theorem 3 in~\cite{BK2005} is basically our Corollary~\ref{Cor:Sum_solutions}
under the assumptions~\eqref{eq:A4b} and $\Prob(N < \infty)=1$.
The first complete description of $\SolL$ in the case of the existence
of an integrable solution $W$ to~\eqref{eq:SumFP} together with a
criterion for the occurrence of the latter is due to Iksanov~\cite{Iks2004}, Proposition 3(a) and (c).

The rest of this paper is organised as follows. In Section \ref
{sec:simple_inclusions}, we prove the simple inclusions in Theorem \ref
{Thm:Min_solutions} and Corollary~\ref{Cor:Sum_solutions}. Sections
\ref{sec:WBR}--\ref{sec:Proof_disintegration} are dedicated to the
proof of the converse direction of these two results. As indicated in
the introduction, iteration of~\eqref{eq:FE1} naturally leads to a
branching model (variously known as weighted branching, branching
random walk and multiplicative cascade)\vadjust{\goodbreak} which we formally define in
Section~\ref{sec:WBM}.
Section~\ref{sec:end_FP} is devoted to the property of endogeny,
mentioned earlier and first introduced in~\cite{AB2005}.
Section~\ref{sec:WBR} collects some (known) connections between the
branching model and random walk theory. The key object derived from
solutions to~\eqref{eq:FE1}, called their disintegration, is described
in Section~\ref{sec:disintegration}. With the help of this notion we
are able to formulate a further result (Theorem \ref
{Thm:Disintegration}) from which the proofs of Theorem \ref
{Thm:Min_solutions} and Corollary~\ref{Cor:Sum_solutions} are easily
completed. Section~\ref{sec:GBP} contains auxiliary results from the
theory of general (CMJ) branching processes. The assertions on slow
variation of $D$ and of $D_\alpha$ are then proved in Sections \ref
{sec:unique} and~\ref{sec:Reg_var_in_0}, respectively. Based
on these results, we prove Theorem~\ref{Thm:essential_uniqueness_and_reg_var}
(Section~\ref{sec:unique}) and Theorem~\ref{Thm:Disintegration}
(Section~\ref{sec:Proof_disintegration}). The final section briefly
addresses the possibility of nonmonotonic solutions to~\eqref{eq:FE1}.

\section{The simple inclusions} \label{sec:simple_inclusions}

\begin{Lemma} \label{Lem:simple_inclusions}
Let~\eqref{eq:A1}--\eqref{eq:A3} and~\eqref{eq:E} hold.
Then $f \in\SolM$ for any $f$ which is defined by \eqref
{eq:F_inf=W_Lambda_alpha}. If, moreover, $\alpha\leq1$ and the
parameter function $h$ in~\eqref{eq:F_inf=W_Lambda_alpha} is chosen
from $\mathfrak{P}_r$, then $f \in\SolL$.
\end{Lemma}

\begin{pf}
Since $W$ satisfies~\eqref{eq:sm-fe} and $h(t) = h(tT_i)$ a.s.\ for $h
\in\mathfrak{H}_r$,
\begin{eqnarray*}
f(t)
&=&
\varphi(h(t) t^\alpha)
=
\E\exp(- W h(t) t^\alpha) \\
&=& \E\exp\biggl(-\sum_{i \geq1} T_i^{\alpha} W_i h(t) t^\alpha\biggr) \\
& =&
\E \biggl(\E \biggl[ \prod_{i \geq1} \exp(-W_i h(t T_i) (tT_i)^\alpha) | T \biggr]
\biggr) \\
&=& \E\prod_{i \geq1} \varphi(h(tT_i) (tT_i)^{\alpha}) =
\E\prod_{i \geq1} f(t T_i).
\end{eqnarray*}
Therefore, $f$ solves the functional equation. Then it is easily
verified that $f \in\SolM$. Now, moreover, suppose that $h \in
\mathfrak{P}_r$. Then $f(t) = \varphi(h(t)t^{\alpha}) \in\mathcal
{L}$ by~\cite{Fel1971}, Criterion 2 on page 441, and Bernstein's theorem.
\end{pf}

\section{The associated branching model} \label{sec:WBM}

A key tool for the further analysis of equation~\eqref{eq:FE1} is an
associated
weighted branching model (or multiplicative cascade, or branching
random walk) which arises upon iteration of~\eqref{eq:FE1} and which
we now describe.

Let $\V:= \bigcup_{n \in\N_0} \N^n$ be the infinite Ulam--Harris
tree, where $\N:= \{1,2,\ldots\}$ and $\N^0 = \{\varnothing\}$.
Abbreviate $v = (v_1,\ldots,v_n)$ by $v_1 \ldots v_n$ and write $v|k$
for the restriction of $v$ to the first $k$ entries, that is, $v|k :=
v_1 \ldots v_k$, $k \leq n$. If $k > n$, put $v|k := v$. Write $vw$ for
the vertex $v_1 \ldots v_n w_1 \ldots w_m$ where $w = w_1 \ldots w_m$.
In this situation, we say that $v$ is an ancestor of $vw$. The length
of a node~$v$ is denoted by $|v|$, thus $|v|=n$ iff $v\in\N^{n}$.
Next, let $\bT:= (T(v))_{v \in\V}$ denote a~family of i.i.d.\ copies\vadjust{\goodbreak}
of $T$, where $T(\varnothing) = T = (T_i)_{i \geq1}$. We interpret
$T_i(v)$ as a weight attached to the edge $(v,vi)$ in the infinite tree
$\V$ and then define $L(\varnothing) := 1$ and, recursively,
$L(vi) := L(v) T_i(v)$ for $v \in\V$ and $i \in\N$. Thus~$L(v)$ is
the product of the weights along the unique path from the root
$\varnothing$ to $v$.
With this branching model, $n$fold iteration of~\eqref{eq:FE1} gives
\begin{equation} \label{eq:FEn}
f(t) = \E\prod_{|v|=n} f(tL(v))\qquad  (t \geq0).
\end{equation}
For $n\in\N_0$, let $\A_n$ denote the $\sigma$-algebra generated by
the sequences $T(v)$, $|v|<n$ and put $\A_{\infty}:= \sigma(\A_n\dvtx n
\geq0)=\sigma(T(v)\dvtx v\in\V)$.
For $\theta\geq0$, define
\begin{equation} \label{eq:W^gamma_n}
W^{(\theta)}_n := \sum_{|v|=n} L(v)^{\theta},\qquad  n \geq0.
\end{equation}
Then $N_n := W^{(0)}_{n} = \sum_{|v|=n} \1_{\{L(v)>0\}}$ counts the
positive branch weights in generation $n$. If $N = N_1 < \infty$ a.s.,
then $(N_n)_{n \geq0}$ forms an ordinary Galton--Watson process with
offspring distribution $\Prob(N \in\cdot)$.
Assuming~\eqref{eq:A3}, and thus $m(\alpha) = 1$, the sequence
$(W_n^{(\alpha)})_{n \geq0}$ is a nonnegative martingale with respect
to $(\A_n)_{n \geq0}$ and hence converges a.s.\
to $W^{(\alpha)}:= \lim_{n \to\infty} W_n^{(\alpha)}$, which satisfies
$\E W^{(\alpha)} \leq 1$ by Fatou's lemma.
Let $\varphi_{\alpha}$ denote its Laplace transform. The martingale
has been studied, in several disguises, by numerous authors. Further
information on $W^{(\alpha)}$ will be given in the next section.

Let us further introduce the shift operators $[\cdot]_u$, $u \in\V$.
Given any function $\Psi=\psi(\bT)$ of the weight family $\bT
=(T(v))_{v \in\V}$ pertaining to $\V$, define
\[
[\Psi]_u := \psi((T(uv))_{v \in\V} )
\]
to be the very same function but for the weights pertaining to the
subtree rooted at $u \in\V$.
Any branch weight $L(v)$ can be viewed as such a function, and thus we have
$[L(v)]_u = T_{v_1}(u) \cdot\,\cdots\,\cdot T_{v_n}(uv_1 \ldots v_{n-1})$
if $v = v_1 \ldots v_n$, that is, $[L(v)]_u = L(uv)/L(u)$ whenever $L(u)>0$.

\section{Endogeny and the smoothing transformation} \label{sec:end_FP}

For our purposes, the relevance of the martingale limit $W^{(\alpha)}$,
defined through~\eqref{eq:W^gamma_n}, with $\alpha$ given by $\eqref
{eq:A3}$, stems from the fact that $W^{(\alpha)}$, unless it is zero
a.s., provides a~$W$ in~\eqref{eq:E} and thus a possible mixing
variable in our main results. In the following we will dwell upon an
additional property associated with $W^{(\alpha)}$, that of endogeny,
introduced by Aldous and Bandyopadhyay~\cite{AB2005}, Definition 7.
This term applies to what they call
a recursive tree process (RTP). Specializing~\cite{AB2005}, equation\
(8), to the situation of equation~\eqref{eq:sm-fe}, suppose
there are random variables $W_u,$ $u \in\V$ with
\begin{equation} \label{endog1}
W_u = \sum_{i \geq1} T_i(u)^{\alpha} W_{ui}\qquad  \mbox{a.s.}
\end{equation}
and that, independent of the first $n-1$ generations in the branching
process, the $\{W_u, |u|=n\}$ are i.i.d. Then $\{W_u\dvtx u \in\V\}$ is an
RTP. Its definition involves the family\vadjust{\goodbreak} $\bT= (T(u))_{u \in\V}$,
sometimes called an \emph{innovation process} in this context.
Note that~\eqref{endog1} implies that
\begin{equation} \label{eq:end_FP}
W_\varnothing= \sum_{|v|=n} L(v)^{\alpha} W_{v} \qquad \mbox{a.s.}
\end{equation}
for all $n \geq0$.
An RTP is called invariant if the $W_u$, $u \in\V$ are identically
distributed.
By Lemma 6 in~\cite{AB2005}, for any distribution $P$ satisfying the
distributional recursion~\eqref{eq:sm-fe} there is an invariant RTP
with marginal distribution~$P$, that is, an RTP $\{W_u\dvtx u \in\V\}$
with $W_u$ having distribution $P$ for all~$u$.\looseness=-1

\begin{Def} \label{Def:endogenous_RTP}
An invariant RTP is called endogenous if there exists a measurable
function $g\dvtx [0,\infty)^{\V} \to[0,\infty]$ such that
$W_\varnothing= g(\bT)$.
An RTP will be called null when $W_u=0$ a.s.---a null RTP is endogenous.
\end{Def}

It is well known that
\[
\bigl[W^{(\alpha)}\bigr]_u = \sum_{i \geq1} T_i(u)^{\alpha} \bigl[W^{(\alpha
)}\bigr]_{ui} \qquad \mbox{a.s.},
\]
and, since $W^{(\alpha)}$ is $\bT$-measurable, this is an endogenous
RTP for equation~\eqref{eq:sm-fe}---but it is interesting only when
not null. Lyons~\cite{Lyo1997} showed that [under \eqref
{eq:A1}--\eqref{eq:A3}] condition~\eqref{eq:A4a} is sufficient for
$W^{(\alpha)}$ to be nondegenerate at $0$. The complete
characterization of the nondegeneracy of $W^{(\alpha)}$ is due to
Iksanov~\cite{Iks2004}. A detailed proof can be found in \cite
{AI2009}. Therefore, under~\eqref{eq:A1}--\eqref{eq:A4}, $W^{(\alpha
)}$ can only be degenerate if~\eqref{eq:A4a} fails and, thus, \eqref
{eq:A4b} holds.\looseness=-1

Even if $W^{(\alpha)} = 0$ a.s., so that the martingale generates a
null RTP, there may be nonnull endogenous RTP. Under suitable
conditions, the limit of the Seneta--Heyde normed version of
$W^{(\alpha)}_n$ (see~\cite{BK1997} and~\cite{HS2009}) will give a
nonnull endogenous RTP. Furthermore, if $m'(\alpha)=0$, under
additional moment conditions, the so-called derivative martingale
converges a.s.\ to a nondegenerate random variable $\partial W^{(\alpha
)}$ which again gives a nonnull RTP; see~\cite{BK2005}, page~623f. In
fact, we will show that under~\eqref{eq:A1}--\eqref{eq:E} there is
always a~nonnull endogenous RTP.

\begin{Theorem} \label{Thm:endogeny}
Assuming~\eqref{eq:A1}--\eqref{eq:E}, the following assertions hold true:

\begin{longlist}[(a)]
\item[(a)]
There exists a nonnull endogenous RTP $\{W_u\dvtx u \in\V\}$ for equation~\eqref{eq:sm-fe}, namely,
\begin{equation} \label{eq:form_of_the_RTP}
W_\varnothing = \lim_{n \to\infty} \sum_{|v|=n}
\bigl(1-\varphi(L(v))\bigr) \qquad \mbox{a.s., }
W_u = [W_{\varnothing}]_u, u \in\V.
\end{equation}
Any other nonnegative invariant RTP for equation~\eqref{eq:sm-fe} is a
scale multiple of this one.
\item[(b)]
Unless $\alpha=1$, there is no nonnull endogenous RTP for equation
\eqref{eq:SumFP}.
\end{longlist}
\end{Theorem}

This theorem solves Open Problem 18 in~\cite{AB2005}: part (a) extends
Corollary~17 in~\cite{AB2005} by imposing weaker moment conditions and
also dealing with\vadjust{\goodbreak} the case \mbox{$m'(\alpha)=0$} [corresponding to $\rho'(1)
= 0$ there], while part (b) states that any endogenous RTP for \eqref
{eq:SumFP} must be null and thus trivial if $\alpha< 1$. Similar
assertions concerning endogeny for two-sided solutions to \eqref
{eq:SumFP} and~\eqref{eq:sm-fe} can be found in~\cite{AM2010b}, Section~4.8.
We postpone the proof of this result until the end of Section \ref
{sec:Proof_disintegration}.
Some partial results relating to Theorem~\ref{Thm:endogeny} that we
need will now be given as propositions.

\begin{Prop} \label{Prop:integrable_endogenous_FP}
Suppose that~\eqref{eq:A1}--\eqref{eq:A3} and~\eqref{eq:E} hold,
with associated RTP $\{W_u\dvtx u \in\V\}$,
and suppose that $\E W=c \in(0,\infty)$. Then $\E W^{(\alpha)}=1$
and $\{W_u\dvtx u \in\V\}$ is endogenous and given by $W_v = c[W^{(\alpha
)}]_v$ a.s. for all $v \in\V$. If, furthermore,~\eqref{eq:A4} holds
true, then condition~\eqref{eq:A4a} is satisfied.
\end{Prop}
\begin{pf}
By~\eqref{eq:end_FP}, the integrability of $W$ and the martingale
convergence theorem,
\begin{eqnarray*}
\E(W_\varnothing| \A_\infty)
& = &
\lim_{n \to\infty} \E(W_\varnothing | \A_n)
= \lim_{n \to\infty} \E\biggl(\sum_{|v|=n} L(v)^{\alpha} W_v | \A_n \biggr)
\\
& = &
c \lim_{n \to\infty} W_n^{(\alpha)}
= c W^{(\alpha)} \qquad \mbox{a.s.},
\end{eqnarray*}
and taking expectations shows that $\E W^{(\alpha)}=1$. Now, for
arbitrary $n \in\N$,
\[
W_\varnothing-cW^{(\alpha)}= \sum_{|v|=n} L(v)^{\alpha}
\bigl(W_{v}-c\bigl[W^{(\alpha)}\bigr]_v \bigr)
\]
and the method of proof in
\cite{AB2005}, Corollary 17, shows $W_v=c[W^{(\alpha)}]_v$ a.s. for
all $v$, which proves the first part of the proposition.

Now suppose additionally that~\eqref{eq:A4} holds true. Let $S_1$ have
the distribution
\begin{equation}\label{eq:mu-alpha}
\Prob(S_1 \in B) := \mu_{\alpha}(B) := \E\sum_{i=1}^N
T_{i}^{\alpha}\1_{B}(-\log T_{i})
\end{equation}
for Borel subsets $B$ of $\R$. Note that, by~\eqref{eq:A4}, the
definition of $\alpha$, and of $m'(\alpha)$,
$\E S_1 =m'(\alpha) \in[0,\infty)$. Now~\cite{Lyo1997} implies
\eqref{eq:A4a} holds.
\end{pf}

We finish the section with a uniqueness result that sharpens Theorem
\ref{Thm:essential_uniqueness_and_reg_var} in the case of endogenous
RTP. Note that, in contrast to Proposition \ref
{Prop:integrable_endogenous_FP}, here it is assumed that the RTP is endogenous.

\begin{Prop} \label{Prop:uniqueness_of_end_FP}
Assume~\eqref{eq:A1}--\eqref{eq:E}.
Let $\widehat{W}$ be another nonnegative variable, but with Laplace
transform $\widehat{\varphi}$, satisfying~\eqref{eq:E}. Suppose
there are corresponding endogenous RTPs $\{W_u\dvtx u \in\V\}$ and $\{
\widehat{W}_u\dvtx u \in\V\}$ with respect to the same innovations
process $\bT$. Then $W_\varnothing= c\widehat{W}_\varnothing$ a.s.\
for some $c > 0$.
\end{Prop}

\begin{pf}
By Theorem~\ref{Thm:essential_uniqueness_and_reg_var}, we already know
that $\varphi(t)=\widehat{\varphi}(ct)$, and it is no loss of
generality to assume $c=1$. Using endogeny, the bounded and thus\vadjust{\goodbreak}
integrable random variable $\exp(-W_\varnothing)$ can be written in
the form
\begin{eqnarray*}
\exp(-W_\varnothing) &=& \E ( \exp(-W_\varnothing) | \A_\infty
)\\
&=& \lim_{n \to\infty} \E\biggl ( \exp \biggl(- \sum_{|v|=n} L(v)^{\alpha}
W_v \biggr)
\Big| \A_n \biggr) \\
&=& \lim_{n \to\infty} \prod_{|v|=n} \varphi(L(v)^{\alpha}) \qquad \mbox{a.s.},
\end{eqnarray*}
and a similar result holds for $\exp(-\widehat{W}_\varnothing)$ with
$\widehat{\varphi}$ instead of $\varphi$ on the right-hand side. Now
$\varphi= \widehat{\varphi}$ implies $\exp(-W_\varnothing) = \exp
(-\widehat{W}_\varnothing)$ a.s.
\end{pf}

This result is first used in the proof of Theorem \ref
{Thm:Disintegration} in Section~\ref{sec:Proof_disintegration}. The
only ingredient to the proof of the previous result which has not yet
been verified is Theorem~\ref{Thm:essential_uniqueness_and_reg_var},
and that will be proved in Section~\ref{sec:unique}, so there is no
circularity in the argument.

\section{Renewal arguments} \label{sec:WBR}

Let $(S_n)_{n \geq0}$ denote a zero-delayed random walk with increment
distribution $\mu_{\alpha}$ introduced at~\eqref{eq:mu-alpha}.
Let $S(v) := -\log L(v)$ ($v \in\V$) where $-\log0 := \infty$.
It is then easily verified (see~\cite{BK1997}, Lemma 4.1) that
\begin{equation} \label{eq:connection}
\Prob(S_n \in\cdot) = \mu_{\alpha}^{*n} = \E\sum_{|v|=n}
e^{-\alpha S(v)} \delta_{S(v)}\qquad (n \in\N_0).
\end{equation}
Importantly, this connection between the branching model
and its associated random walk is preserved under certain stopping schemes.
To make this precise in the present context, let $\sigma\dvtx  \R^{\N_0}
\to\N_{0} \cup\{\infty\}$ denote a formal stopping rule, that is,
\[
\sigma((s_n)_{n \geq0}) = \inf\{n \geq0\dvtx  (s_0,\ldots,s_n) \in B_n\}
\]
where $B_n$ is a Borel subset of $\R^{n+1}$, $n \geq0$. For $n \in\N
_0$, let $\sigma_n$ denote the $n$th consecutive application of
$\sigma$, which means that $\sigma_0 := 0$ and
\[
\sigma_n := \inf\{k > \sigma_{n-1}\dvtx  (0, s_{\sigma_{n-1} + 1} -
s_{\sigma_{n-1}},\ldots, s_k-s_{\sigma_{n-1}})\in B_{k-\sigma_{n-1}}\}
\]
for $n \in\N$. Then, for any $x = (v_i)_{i \geq1} \in\N^{\N} =:
\partial V$, the boundary of the Ulam--Harris tree $\V$, we can apply
these formal stopping rules to the random walk along the infinite path
$\varnothing\to v_1 \to v_1 v_2 \to\cdots$ from the root to the boundary
of $\V$; that is, we can consider $\sigma_n((S(x|k))_{k \geq0})$, $n
\in\N_0$. The set of all vertices in $\V$ in which $\sigma_n$ stops
any random walk from the root to the boundary of~$\V$ is denoted by
$\mathcal{T}_{\sigma_n}$, that is,
\[
\mathcal{T}_{\sigma_n} := \{x|\sigma_n((S(x|k))_{k \geq0})\dvtx  x \in
\partial\V\}.
\]
We refer to the (random) sets $\mathcal{T}_{\sigma_n}$ as \emph
{homogeneous stopping lines} (\textit{HSLs}). This notion indicates that the
above defined random sets are special optional lines in the sense of\vadjust{\goodbreak}
Jagers~\cite{Jag1989}, Kyprianou~\cite{Kyp2000} and Biggins and
Kyprianou~\cite{BK2004}, but where, additionally, stopping along any
path of the infinite tree $\V$ follows the same stopping rule.
By some obvious changes in the proof of Lemma 3.2 in~\cite{AM2008}, we
deduce that
\begin{equation} \label{eq:embedded_connection}
\E\sum_{v \in\mathcal{T}_{\sigma_n}} e^{- \alpha S(v)} \delta_{S(v)}
=
\Prob(S_{\sigma_n} \in\cdot, \sigma_n < \infty)
=: (\mu_{\alpha}^{\sigma})^{*n},
\end{equation}
where in slight abuse of notation we write $\sigma_n$ instead of
$\sigma_n((S_{k})_{k \geq0})$. We have thus established the analogue
of~\eqref{eq:connection} for the embedded branching model based upon
$(\sigma_{n})_{n\ge0}$. Here we make use of the HSLs associated with
the first ascending ladder epoch defined by $\sigma^{\st}:=\inf\{k
\geq0\dvtx  s_k > 0\}$. When applied to $(S_n)_{n\ge0}$, this ladder epoch
will again be denoted by $\sigma^{\st}$, whereas~$\mu_{\alpha
}^{\sigma^{\st}}$ will be abbreviated to $\mu_{\alpha}^{\st}$.

\begin{Lemma} \label{Lem:m}
If~\eqref{eq:A1}--\eqref{eq:A4} hold, then $\limsup_{n \to\infty}
S_n = \infty$ a.s.\ and $\sigma^{\st} < \infty$ a.s.
\end{Lemma}
\begin{pf}
Under~\eqref{eq:A4} $\E S_1 \geq0$ and the result follows from
standard random walk theory.
\end{pf}

\begin{Lemma} \label{Lem:first_ladder_height_moment}
If~\eqref{eq:A1}--\eqref{eq:A3} and~\eqref{eq:A4a} hold, then $\E
S_{\sigma^{\st}} < \infty$.
\end{Lemma}
\begin{pf}
The first part of~\eqref{eq:A4a} is equivalent to $\E S_1 \in
(0,\infty)$. Thus, from standard random walk theory, we infer
integrability of $\sigma^{\st}$ and then that $\E S_{\sigma^{\st}}
= \E\sigma^{\st} \E S_1 < \infty$ by Wald's equation.
\end{pf}

\begin{Lemma}[(c.f.~\cite{BK2005}, Theorem 10(c))] \label
{Lem:exponential_moments}
If~\eqref{eq:A1}--\eqref{eq:A3} hold, then, for any $0\le\theta\le
\alpha$,
\[
\E\sum_{v \in\mathcal{T}_{\sigma^{\st}}} L(v)^{\theta} < \infty
\quad \mbox{if, and only if,}\quad
\E\sum_{i \geq1} T_i^{\theta} < \infty.
\]
\end{Lemma}

\begin{pf}
Using~\eqref{eq:connection},~\eqref{eq:embedded_connection} and
$\Prob(\sigma^{\st} < \infty) = 1$, we infer that the result is
equivalent to the assertion
\[
\E e^{(\alpha-\theta)S_{\sigma^{\st}}} < \infty
\quad \mbox{if, and only if,}\quad
\E e^{(\alpha-\theta)S_1} < \infty,
\]
which in turn can be deduced from results in standard random walk
theory, see, for instance,~\cite{Fel1971}, Section XII.3.
\end{pf}

\section{Disintegration} \label{sec:disintegration}

Our analysis of equation~\eqref{eq:FE1} will be built on a pathwise
counterpart of~\eqref{eq:FEn}. Let
\begin{equation} \label{eq:disintegrated}
M_n(t) := \prod_{|v|=n} f(t L(v)) ,\quad  n \geq0
\end{equation}
for $f \in\SolM$. Neveu~\cite{Nev1988} studied the \emph
{multiplicative martingales} $(M_n(t))_{n \geq0}$ in the context of
the KPP equation. More recently, they have been considered in the study\vadjust{\goodbreak}
of the functional equation of the smoothing transform \cite
{BK1997,BK2005}. We state the fact that $(M_n(t))_{n \geq0}$ is indeed
a martingale in the following lemma~\cite{BK1997}, Theorem~3.1.

\begin{Lemma}\label{Lem:Disintegration}
Let $f \in\SolM$ and $t \geq0$. Then $(M_n(t))_{n \geq0}$ forms a
bounded nonnegative martingale with respect to $(\A_n)_{n \geq0}$
and thus converges a.s.\ and in mean to a random variable $M(t)$ satisfying
\begin{equation} \label{eq:Disintegration_integrated}
\E M(t) = f(t).
\end{equation}
\end{Lemma}

In the situation of Lemma~\ref{Lem:Disintegration}, we call the
stochastic process $M=(M(t))_{t \geq0}$ the \emph{disintegration of
$f$} (w.r.t.\ $T$) and also a \emph{disintegrated fixed point}. By
Lemma~\ref{Lem:Disintegration}, we can calculate any solution to the
functional equation~\eqref{eq:FE1} from its associated disintegrated
fixed point.
\begin{Def} \label{Def:endogenous_FP}
We say that a random variable $W$ is an \emph{endogenous fixed point
w.r.t. $T^{(\alpha)}$} if $W$ is as in~\eqref{eq:E} and if there
exists an endogenous RTP $\{W_u\dvtx u \in\V\}$ such that $W =
W_{\varnothing}$.
\end{Def}

\begin{Theorem} \label{Thm:Disintegration}
If~\eqref{eq:A1}--\eqref{eq:E} hold, then for any $f \in\SolM$ with
disintegration $M$ there is a function $h \in\mathfrak{H}_r$ such that
\begin{equation} \label{eq:Disintegration}
M(t) = e^{- W h(t)t^{\alpha}} \qquad \mbox{a.s. }  (t \geq0)
\end{equation}
where $W$ is an endogenous fixed point w.r.t.\ $T^{(\alpha)}$.
\end{Theorem}

The proof of this theorem is postponed until Section \ref
{sec:Proof_disintegration}. The result is the first that provides a
full description of the set of disintegrations of the functions from
$\SolM$.
It is, as mentioned just after Corollary~\ref{Cor:Sum_solutions}, our
central result.
A similar result is implicit in the proof of Theorem 4.2 in \cite
{AM2009} but covers only disintegrations of sufficiently regular $f \in
\SolM$.
Theorem~\ref{Thm:Disintegration} has great impact on the analysis of
fixed points of inhomogeneous smoothing transforms~\cite{AM2010a}, Theorems 4.4
and~8.1, as well as of two-sided fixed points of the smoothing
transform~\cite{AM2010b}, Section~4.5 and Proposition 5.3.
Next we show how it allows us to complete the proofs of Theorem~\ref
{Thm:Min_solutions} and Corollary~\ref{Cor:Sum_solutions}.

\begin{pf*}{Proof of Theorem~\ref{Thm:Min_solutions}}
By Lemma~\ref{Lem:simple_inclusions}, we have $f \in\SolM$ for any~$f$
given by~\eqref{eq:F_inf=W_Lambda_alpha} and parametrized with $h
\in\mathfrak{H}_r$.
For the reverse inclusion, pick any $f \in\SolM$.
Theorem~\ref{Thm:Disintegration} shows the existence of an endogenous
fixed point $W$ w.r.t.\ $T^{(\alpha)}$ and an $h \in\mathfrak{H}_r$
such that the disintegration $M$ of~$f$ satisfies \eqref
{eq:Disintegration}. This in combination with \eqref
{eq:Disintegration_integrated} gives $f(t) = \varphi(h(t) t^{\alpha
})$ for $t > 0$, as required.
\end{pf*}

\begin{pf*}{Proof of Corollary~\ref{Cor:Sum_solutions}}
Let $\alpha\le1$. Again, Lemma~\ref{Lem:simple_inclusions} gives one
inclusion. For the reverse one, pick any $f \in\SolL$. As in the
proof of Theorem~\ref{Thm:Min_solutions}, we obtain $f(t) = \varphi
(h(t)t^{\alpha})$ a.s.\ ($t \geq0$) for some $h \in\mathfrak{H}_r$.
It remains to show that $h \in\mathfrak{P}_r$. To this end, it
suffices to show that $t \mapsto h(t)t^{\alpha}$ has a completely\vadjust{\goodbreak}
monotone derivative in the $r$-geometric case. Without loss of
generality, we assume $h(1) = 1$ and use the regular variation of
$1-\varphi$ (see Theorem~\ref{Thm:essential_uniqueness_and_reg_var})
to infer
\[
\frac{1-f(t r^{-n})}{1-f(r^{-n})}
=
\frac{1-\varphi(h(t)t^{\alpha} r^{-\alpha n})}{1-\varphi(r^{-\alpha n})}
\to
h(t) t^{\alpha}\qquad  (n \to\infty).
\]
Thus $t \mapsto h(t)t^{\alpha}$ is the limit of a sequence of
functions with completely monotone derivatives and therefore has a
completely monotone derivative itself.
\end{pf*}

\begin{pf*}{Proof of Corollary~\ref{Cor:inf-solutions}}
Let $g$ be the generating function of the family size $N$. From \eqref
{eq:MinFP},
$\Prob(X=\infty) = g(\Prob(X=\infty))$ and $\Prob(X>0) \leq
g(\Prob(X>0))$.
Since $X$ is nondegenerate $\Prob(X=\infty)<\Prob(X>0)\leq1$,
which implies that $\Prob(X>0) \geq g(\Prob(X>0))$. Consequently
$\Prob(X>0)$ is another a fixed point of~$g$ and so must equal one.
Thus the survival function $f(t)=\Prob(X\geq t)$ has $f(0+)=1$ and so
$ f \in\mathcal{M}$.
The result now follows from Theorem~\ref{Thm:Min_solutions}.
\end{pf*}

We finish this section with a series of results that will be useful in
the proof of Theorem~\ref{Thm:Disintegration}.

\begin{Lemma}[(see Lemma 5.2 in~\cite{AM2009})] \label{Lem:disintSFPE}
Let $f \in\SolM$ with disintegration~$M$. Then, for all $t \geq0$
and $n \in\N_0$, we have
\begin{equation} \label{eq:DisintegratedFPE}
M(t) = \prod_{|v|=n} [M]_v(t L(v))\qquad  \mbox{a.s.}
\end{equation}
\end{Lemma}

Lemma~\ref{Lem:disintSFPE} provides us with a quick proof of the fact
that $\SolM$ is contained in the set of solutions to the functional
equation~\eqref{eq:FE1} with the sequence $T$ replaced by the family
$(L(v))_{v \in\mathcal{T}}$, where $\mathcal{T}$ is an \emph{a.s.\
dissecting HSL}. The last notion was introduced in~\cite{Kyp2000} for
general stopping lines. For a HSL $\mathcal{T}$ it means that a.s.\
there exists a positive integer $n$ such that for any $v \in\N^n$
there is some $u \in\mathcal{T}$ satisfying $u = v|k$ for some $k <
|v|$. In other words, with probability one $\mathcal{T}$ cuts through
the tree prior to some (random) generation $n$.

\begin{Lemma} \label{Lem:HSL_FPE}
Let $f \in\SolM$ with disintegration $M$ and let $\mathcal{T}$
denote an a.s.\ dissecting HSL. Then
\[
M(t) = \prod_{v \in\mathcal{T}} [M]_v(t L(v))
\qquad \mbox{a.s.}
\]
and thus
\[
f(t) = \E\prod_{v \in\mathcal{T}} f(t L(v))\qquad  (t \geq0).
\]
In particular, any $f \in\SolM$ is also a solution to~\eqref{eq:FE1}
with the sequence $(T_i)_{i \geq1}$ replaced by the family $(L(v))_{v
\in\mathcal{T}}$.
\end{Lemma}

The proof of Lemma~\ref{Lem:HSL_FPE} also works (after some minor
changes) for the more general \emph{very simple lines} defined in
\cite{BK2004}, Section 6. These are stopping lines where for any $v
\in\V$ whether $v$ is on the line or not is determined by the
ancestry of $v$, but along different ancestral lines the stopping rules
may be different.
\begin{pf*}{Proof of Lemma~\ref{Lem:HSL_FPE}}
Let $\mathcal{T}$ denote an a.s.\ dissecting HSL and fix $t \geq0$. Define
$B$ to be the set where $[M]_v(t L(v)) = \prod_{i \geq1} [M]_{vi}(t
L(vi))$ for all $v \in\V$.
In view of equation~\eqref{eq:DisintegratedFPE}, the invariance of
$\Prob(\bT\in\cdot)$ under the shift $[\cdot]_v$ and the
independence of $[\bT]_v$ and $L(v)$, we have $\Prob(B) = 1$. Since
$\mathcal{T}$ is a HSL, there exists some formal stopping rule $\sigma
$ such that $\mathcal{T} = \mathcal{T}_{\sigma}$. Putting $\mathcal
{T}_n := \mathcal{T}_{\sigma\wedge n}$ we have that $\mathcal{T}_n$
is the HSL where along each path from the root to the boundary the
stopping vertices are chosen according to the stopping rule $\sigma
\wedge n$. By induction over $n$, we infer that on $B$
\[
M(t) = \prod_{v \in\mathcal{T}_n} [M]_v(t L(v))
\]
for all $n \geq0$. Passing to the limit $n \to\infty$ yields the
assertion since $\mathcal{T}$ is a.s.\ dissecting so that $\mathcal
{T}=\mathcal{T}_n$ for some (random) $n$.
\end{pf*}

Now we wish to approximate a disintegrated fixed point $M$ not only by
the sequence $M_n(t)$, $n \geq0$, which takes the product over a fixed
generation, but also by terms like $M_{\mathcal{T}}(t)$, where the
product is taken over all vertices in a HSL $\mathcal{T}$. Here, as in
\cite{BK1997}, we focus on special HSLs, namely, first exit lines~$\mathcal{T}_t$ based on the first exit times $\tau(t)$, viz.\ $\tau
(t):=\inf\{k\ge0\dvtx s_{k}>t\}$ and
\[
\mathcal{T}_t := \mathcal{T}_{\tau(t)}
= \{v \in\V\dvtx  S(v) > t \mbox{ and } S(v|k) \leq t \mbox{ for }
k=0,\ldots,|v|-1\}.
\]

\begin{Lemma} \label{dissecting}
Assume~\eqref{eq:A2} and~\eqref{eq:A3} hold. Then
\textup{(a)} $\sup_{|v| \geq n} L(v) \to0$ a.s., and \textup{(b)} $\mathcal{T}_t$
is dissecting a.s.
\end{Lemma}
\begin{pf} By Theorem 3 in~\cite{Big1998},
$\sup_{|v| = n} L(v) \to0$ a.s., which implies the first assertion.
This is equivalent to $\inf_{|v| \geq n}S(v) \to\infty$ a.s. Thus
there is a (random) $n(t)$ such that
$\inf_{|v| \geq n(t)}S(v)>t$ and then every $v \in\mathcal{T}_t$ has
$|v| \leq n(t)$.
\end{pf}

\begin{Lemma} \label{Lem:disintegration_via_HSL}
Given $f \in\SolM$ with disintegration $M$, the following assertions
hold for all $t\ge0$:
\begin{longlist}[(a)]
\item[(a)] $ {\lim_{n \to\infty} \sum_{|v|=n} 1-f(tL(v)) = -\log
M(t)}$ a.s.
\item[(b)] $ \lim_{u \to\infty} \prod_{v \in\mathcal{T}_u}
f(tL(v)) = M(t)$ a.s.
\item[(c)] $ {\lim_{u \to\infty} \sum_{v \in\mathcal{T}_u}
1-f(tL(v)) = -\log M(t)}$ a.s.
\end{longlist}
\end{Lemma}
\begin{pf}
(a)
Using Lemma~\ref{dissecting}(a), $f(0+) = 1$, and $-\log x \sim1-x$
as $x \to1$, we infer for arbitrary $t > 0$
\[
-\log M(t)
 = -\log\lim_{n \to\infty} \prod_{|v|=n} f(tL(v))
 = \lim_{n \to\infty} \sum_{|v|=n} 1-f(tL(v))
\qquad \mbox{a.s.}
\]

(b) For $u \geq0$, denote by $\A_{\mathcal{T}_u} := \sigma(T(v)\dvtx  v
\prec\mathcal{T}_u)$ the pre-$\mathcal{T}_u$ $\sigma$-algebra.
Here, $v \prec V$ for $v \in\V$ and $V \subseteq\V$ means that $v$
has no ancestor in $V$, in particular, $v \notin V$ (see \cite
{Jag1989} for a full discussion). More precisely, $\A_{\mathcal
{T}_u}$ is defined as
\[
\label{eq:A_T}
\A_{\mathcal{T}_u} = \sigma \bigl(\{T(v) \in A \} \cap\{v \prec\mathcal
{T}_u\}\dvtx  v \in\V, A \in\B([0,\infty)^{\N}) \bigr),
\]
where $\B$ denotes the Borel-$\sigma$-algebra. $\A_{\mathcal{T}_u}$
increases as $u$ increases.
Since, by Lemma~\ref{dissecting}(b), $\mathcal{T}_u$ is dissecting,
the proof of Lemma 6.1 in~\cite{BK1997} applies in the current context
to give
\[
M_{\mathcal{T}_u}(t) := \prod_{v \in\mathcal{T}_u} f(tL(v)) = \E
[M(t) | \A_{\mathcal{T}_u}]\qquad
\mbox{a.s.}
\]
Now let $\mathcal{G} := \sigma(\A_{\mathcal{T}_u}\dvtx  u \geq0)$.
Standard theory implies that $M_{\mathcal{T}_u}(t) \to\E[M(t) |
\mathcal{G}]$ a.s.\ as $u \uparrow\infty$. It remains to show that
$M(t)$ is measurable w.r.t.\ $\mathcal{G}$. Since~$M(t)$ is a function
of the weight ensemble $(L(v))_{v \in\V}$, it suffices to show that
any $L(v)$, $v \in\V$ is $\mathcal{G}$-measurable. To this end, fix
$v = v_1 \ldots v_n \in\N^n$. If $L(v) = 0$ and thus $S(v) = \infty
$, we have $v \not\prec\mathcal{T}_{u}$ for all $u \geq0$. If, on
the other hand, $L(v) > 0$, then $v \in\mathcal{T}_u$ for all $u>\max
_{k=0,\ldots,n} S(v|k)$. In both cases, $L(v) = \lim_{u \to\infty}
L(v) \1_{\{v \prec\mathcal{T}_u\}}$. For any fixed $u$,
\[
L(v) \1_{\{v \prec\mathcal{T}_u\}}
= \1_{\{v \prec\mathcal{T}_u\}} \prod_{k=0}^{n-1} T_{v_{k+1}}(v|k) \1
_{\{v|k \prec\mathcal{T}_u\}}.
\]
Clearly, $\1_{\{v \prec\mathcal{T}_u\}}$ is $\A_{\mathcal
{T}_u}$-measurable. Elementary arguments further show that the
$T_{v_{k+1}}(v|k) \1_{\{v|k \prec\mathcal{T}_u\}}$ are also $\A
_{\mathcal{T}_u}$-measurable. Thus, $M(t)$ is $\mathcal
{G}$-measurable. Finally, we should remark that the formulation of the
convergence in Lemma~\ref{Lem:disintegration_via_HSL} indicates that
the convergence holds outside a $\Prob$-null set for any sequence $u
\uparrow\infty$. This is indeed true, for it can be shown that the
martingale $(M_{\mathcal{T}_u}(t))_{u \geq0}$ a.s.\ has
right-continuous paths. (This follows basically from the fact that
a.s.\ the positions $S(v)$, $v \in\V$ do not accumulate in finite
intervals $(a,b)$, $-\infty< a < b < \infty$.) Since we only need
convergence along a fixed subsequence in what follows, we omit further details.

(c) This follows by combining assertion (b) with the arguments given
in~(a), where the simple observation that $L(v) \leq e^{-u}$ for any $v
\in\mathcal{T}_u$ replaces the use of Lemma~\ref{dissecting}(a).
\end{pf}

\begin{Lemma} \label{Lem:Identify_end_FP_and_disintegration}
Let $f \in\SolM$ with disintegration $M$. Suppose further that $1-f$
is regularly varying of index $\alpha$ at $0$ in the nongeometric
case, while in the $r$-geometric case $(1-f(ut))/(1-f(t)) \to u^{\alpha
}$ whenever $u \in r^{\Z}$ and $t$ approaches $0$ through a fixed
residue class $sr^{\Z}$, $s > 0$. Then the\vadjust{\goodbreak} following assertions hold:
\begin{longlist}[(a)]
\item[(a)]
$W_t := -\log M(t)$
is an endogenous fixed point w.r.t.\ $T^{(\alpha)}$ for any $t>0$.
\item[(b)]
If $1-f$ is regularly varying of index $\alpha$ at $0$, then
$M(t)=e^{-t^{\alpha}W_1}$ a.s.\ for all $t \geq0$,
and~\eqref{eq:E} holds with $W=W_1$.
\end{longlist}
\end{Lemma}

\begin{pf}
(a)
Fix $t > 0$ and let $W_t := -\log M(t)$. By the proof of Lem\-ma~6.2 in
\cite{AM2009}, $\E M(t) = f(t) < 1$ and thus $\Prob(W_t > 0) > 0$.
For any $v \in\V$ and $s \in r^{\Z}$, a combination of Lemma \ref
{dissecting}(a), our assumptions on the behavior of~$f$ at $0$ and
Lemma~\ref{Lem:disintegration_via_HSL}(a) gives
\begin{eqnarray} \label{eq:W*_st}
-\log[M]_v(s t)
& = &
\lim_{n \to\infty} \sum_{|u|=n} 1-f(s t [L(u)]_v)
\\
& = &
\lim_{n \to\infty} \sum_{|u|=n} \frac{1-f(s t [L(u)]_v)}{1-f(t
[L(u)]_v)} \bigl(1-f(t [L(u)]_v)\bigr)  \nonumber\\
& = &
s^{\alpha} \lim_{n \to\infty} \sum_{|u|=n} 1-f(t[L(u)]_v) \nonumber
\\
& = &
s^{\alpha} (-\log[M]_v(t))
\qquad \mbox{a.s.} \label{eq:u^alpha_W*_t}
\end{eqnarray}
Use this with $s=L(v)$ for $|v|=n$, and recall \eqref
{eq:DisintegratedFPE} to obtain
\begin{eqnarray*}
W_t
& = &
-\log\prod_{|v|=n} [M]_v(tL(v)) \\
& = &
\sum_{|v|=n} -\log[M]_v(tL(v))
= \sum_{|v|=n} L(v)^{\alpha} [W_t]_v
\qquad \mbox{a.s.,}
\end{eqnarray*}
where in the $r$-geometric case $L(v) \in r^{\Z}$ a.s.,\ for all $v
\in\V$ has been utilized. We have thus proved that $W_t$ is an
endogenous fixed point.

(b)
By an application of equations~\eqref{eq:W*_st} and \eqref
{eq:u^alpha_W*_t}, which are valid for all
$s>0$ if $1-f$ is regularly varying of index $\alpha$ at $0$, we
infer, with $t=1$, $v=\varnothing$, that
\[
M(s) = e^{-s^{\alpha} W_1} \qquad \mbox{a.s.}
\]
for any $s>0$.
Now, using~\eqref{eq:Disintegration_integrated}, $f(t) = \phi
(t^{\alpha})$ for all $t \geq0$ where $\phi$ denotes the Laplace
transform of $W_1$. Therefore $f \in \mathcal{M}$ implies that $\phi
(t) \to1$ as $t \downarrow0$, so that $W_1 < \infty$ a.s.\ and $\phi
(t) < 1$ for $t>0$, so $W_1$ is not identically zero.
Finally, $f \in\SolM$ implies that $\phi$ satisfies~\eqref{eq:FE1}
with $T^{(\alpha)}$ in place of $T$.
\end{pf}

\begin{Lemma} \label{Lem:disintegration varphi}
Let $\varphi$ in~\eqref{eq:E} have disintegration $\Phi$ (w.r.t.\
$T^{(\alpha)}$). If $1-\varphi$ is regularly varying of index $1$ at
$0$, then
$\varphi$ is the Laplace transform of $-\log\Phi(1)$.
\end{Lemma}
\begin{pf}
This follows immediately from Lemma \ref
{Lem:Identify_end_FP_and_disintegration}(b).
\end{pf}

\section{Results for general branching processes} \label{sec:GBP}

The weighted branching mod\-el introduced in Section~\ref{sec:WBM} gives
rise to the definition of a related general (CMJ) branching process.
This is a critical connection here and in~\cite{BK1997,BK2005}.
Recall that $S(v) := -\log L(v)$ for $v \in\V$.
Let $\mathcal{T}^{\st}_n$ denote the HSL associated with the stopping
rule $\sigma^{\st}_n$, the $n$th strictly ascending ladder index
(defined in Section~\ref{sec:WBR}), and let $\mathcal{T}^{\st}$ be
another notation for $\mathcal{T}^{\st}_1$.
The $n$th generation in this general branching process is given by
\[
\mathcal{Z}^{\st}_n := \sum_{v \in\mathcal{T}_n^{\st}} \delta_{S(v)},
\]
where the $S(v)$ occurring here are the birth times of the individuals
in this generation.
The reproduction point process $\mathcal{Z}^{\st}$ of this general
branching process is given by
$\mathcal{Z}^{\st} :=\mathcal{Z}^{\st}_1$. Quantities derived from
$\mathcal{Z}$, like $N$ and $m$, have counterparts for $\mathcal
{Z}^{\st}$ that will be denoted by $N^{\st}$, $m^{\st}$ and so on.
Specifically, let $T^{\st} := (T^{\st}_i)_{i \geq1}$ be the
enumeration of the family $\{L(v)\dvtx  v \in\mathcal{T}^>\}$
in decreasing order where $T_i^{\st} := 0$ if $i> |\mathcal{T}^{\st}|$.
Lemma~\ref{Lem:wlog_T_i<1} below establishes properties of $T^{\st}$
that are inherited from $T$,
or equivalently from the corresponding point process $\mathcal{Z}$,
which was introduced just before~\eqref{eq:A1}.
These properties can easily be reinterpreted as properties of the
reproduction point process $\mathcal{Z}^{\st}$.

\begin{Lemma}[(\textit{cf.} Theorem 10 in~\cite{BK2005} and
Proposition 5.1 in~\cite{AK2005})] \label{Lem:wlog_T_i<1}
If $T$ satisfies~\eqref{eq:A1}--\eqref{eq:A3}, then so does $T^{\st}
= (T^{\st}_i)_{i \geq1}$.
Thus
\begin{eqnarray*}
&\Prob(T^{\st}_i \in\{0,1\} \mbox{ for any } i \geq1) < 1, \qquad \E
N^{\st} > 1 \quad \mbox{and}
\\
&1=m^{\st}(\alpha)<m^{\st}(\beta) \qquad \mbox{for all }\beta\in
[0,\alpha).
\end{eqnarray*}
Moreover, if $T$ satisfies~\eqref{eq:A4a}, then so does $T^{\st}$,
and similarly for~\eqref{eq:A4b}.
Finally, $\G(T)=\G(T^{\st})$, where $\G(T)$ and $\G(T^{\st})$
denote the minimal closed multiplicative subgroups of $\R^+$ generated
by $T$ and $T^{\st}$, respectively.
\end{Lemma}
\begin{pf}
Under the given assumptions, we can apply Lemma~\ref{Lem:m} to infer
that $\Prob(\sigma^{\st} < \infty) = 1$.
Hence Proposition 5.1 in~\cite{AK2005} implies that the sequence
$(T_i^{\st})_{i \geq1}$ satisfies conditions~\eqref{eq:A1}--\eqref{eq:A3}.
Further, if also~\eqref{eq:A4a} is assumed for~$T$, then again
Proposition 5.1 in~\cite{AK2005} yields the validity of \eqref
{eq:A4a} for $T^{\st}$.
If $T$ satisfies~\eqref{eq:A4b}, that is, if $m(\theta) < \infty$
for some $\theta< \alpha$, then Lemma~\ref{Lem:exponential_moments}
yields $m^{\st}(\theta) < \infty$ for the same $\theta$. It remains
to prove that $\G(T) = \G(T^{\st})$. To this end, notice that $-\log
\G(T) = \G(\mu_{\alpha})$ and $-\log\G(T^{\st}) = \G(\mu
_{\alpha}^{\st})$, where $\G(\mu_{\alpha})$ and $\G(\mu_{\alpha
}^{\st})$ denote the minimal closed additive subgroups of $\R$
generated by the distributions $\mu_{\alpha}$ and $\mu_{\alpha
}^{\st}$, respectively. Now, $\mu_{\alpha} = \Prob(S_1 \in\cdot)$
while by equation~\eqref{eq:embedded_connection}, $\mu_{\alpha
}^{\st} = \Prob(S_{\sigma^{\st}} \in\cdot)$. From classical
renewal theory (see, e.g.,~\cite{BD1994}, Section 2) we know that the
minimal closed subgroups generated by a distribution and the associated
ladder height distribution coincide if the associated ladder index is
a.s.\ finite.\vadjust{\goodbreak}
\end{pf}

The key reference for CMJ processes is~\cite{Ner1981}, where $\mu
^{\st}$ is assumed not to be concentrated on a centred lattice (which
corresponds exactly to what is here called the continuous or
nongeometric case) but ``all results could be modified to the lattice
case''~\cite{Ner1981}, page 366.
The details of the lattice case (at least concerning a.s.\ convergence
results) have been supplied in~\cite{Gat2000}.

Keep in mind that $\mathcal{T}_t$ is defined to be the HSL associated
with the first exit time $\tau(t)$. Define $W^{(\alpha)}_{\mathcal
{T}_t} := \sum_{v \in\mathcal{T}_t} L(v)^{\alpha}$. The first
result is just a version of~\cite{Ner1981}, Proposition 2.4.

\begin{Prop} \label{Prop:martingale}
$(W^{(\alpha)}_{\mathcal{T}_t})_{t \geq0}$ is a nonnegative
martingale with\break a.s. limit~$W^{(\alpha)}$.
\end{Prop}

Let $T_t$ be the number of births in the general branching process up
to and including time $t$, that is,
\[
T_t =
|\{v \in\V\dvtx  v \in\mathcal{T}^{\st}_n \mbox{ for some } n \in\N
_0 \mbox{ and } S(v) \leq t\}|.
\]
Let $S$ be the survival set of the process $(N_n)_{n \geq0}$.
Then $S = \{T_t \to\infty\}$ a.s.,\ and $S$ has positive probability
iff $\E N^{\st} > 1$.
Moreover, $S = \{W^{(\alpha)} > 0\}$ a.s.\ if $\Prob(W^{(\alpha)}>0)
> 0$, which is guaranteed by~\eqref{eq:A4a}.

The next result provides us with sufficient conditions for \emph{ratio
convergence} on $S$ of this general branching processes counted by
certain characteristics.
More precisely, it focuses on the asymptotic behavior of the ratio
\begin{equation} \label{eq:ratios}
\frac{\sum_{v \in\mathcal{T}_t} e^{-\beta(S(v)-t)}\1_{\{S(v)-t >
c\}}}{\sum_{v \in\mathcal{T}_t} e^{-\alpha(S(v)-t)}}
\end{equation}
with $\beta>0$. The formulation of the next result is adapted to apply
to both lattice ($r$-geometric) and continuous (nongeometric) cases.

\begin{Prop} \label{Prop:ratios}
Assume~\eqref{eq:A1}--\eqref{eq:A3}, and let $\varepsilon> 0$. Then
the following two assertions hold:
\begin{longlist}[(a)]
\item[(a)]
If~\eqref{eq:A4a} is satisfied, then for $\beta= \alpha$
and all sufficiently large $c$
\begin{equation} \label{eq:ratio_convergence}
\frac{\sum_{v \in\mathcal{T}_t} e^{-\beta(S(v)-t)}\1_{\{S(v)-t >
c\}}}{\sum_{v \in\mathcal{T}_t} e^{-\alpha(S(v)-t)}}
\to \varepsilon(c)
\leq \varepsilon\qquad \mbox{on $S$ as } t \to\infty
\end{equation}
in probability.
\item[(b)]
If~\eqref{eq:A4b} is satisfied, then~\eqref{eq:ratio_convergence}
holds true in the a.s.\ sense for any $\beta\geq\theta$ and all
sufficiently large $c$ (depending on $ \beta$).
\end{longlist}
\end{Prop}

\begin{pf}
The result follows from Theorems 3.1 and 6.3 in~\cite{Ner1981} and the
corresponding lattice-case results
if we check that the appropriate conditions are fulfilled. In what
follows we restrict ourselves to the continuous case, the lattice case
being similar.

First note that in the situation of both assertions (a) and (b), \eqref
{eq:A1}--\eqref{eq:A4} are fulfilled. Thus, by Lemma~\ref{Lem:wlog_T_i<1},
we know that~\eqref{eq:A1}--\eqref{eq:A4} also hold for $T^{\st}$,
and hence, with appropriate translation, for $\mathcal{Z}^{\st}$. The
sums in the ratio in~\eqref{eq:ratio_convergence} are functions of the
BRW $(\mathcal{Z}_n)_{n \geq0}$. Now notice that since in both sums
the summation is over $v \in\mathcal{T}_t$, and the first crossings
of the level $t$ necessarily only occur on vertices that are members of
a strictly increasing ladder line $\mathcal{T}_n^{\st}$, the sums
remain unaffected when replacing the underlying BRW $(\mathcal
{Z}_n)_{n \geq0}$ by the embedded BRW $(\mathcal{Z}^{\st}_n)_{n \geq
0}$. Therefore, by proving the result for the embedded process instead
of the original, it is no loss of generality to assume that $T_i < 1$
for all $i \geq1$, equivalently, $S(v) > 0$ for all $|v|=1$. In this
situation, by~\eqref{eq:A2}, the general branching process $(\mathcal
{Z}_n)_{n \geq0}$ is supercritical. The validity of~\eqref{eq:A3}
implies the existence of a \emph{Malthusian parameter} (viz., $\alpha
$), which is Nerman's condition (ii) in the introduction of~\cite
{Ner1981}, and that of~\eqref{eq:A4} ensures the validity of Nerman's
condition (iii) (this is immediate if~\eqref{eq:A4a} holds whereas it
follows from the fact that, in the given situation, $m$ is strictly
decreasing and convex on $[\theta,\infty)$ in the case that~\eqref
{eq:A4b} holds). Finally, since we are discussing the continuous case,
Nerman's condition (i) is also satisfied.

Now, following Nerman's notation, the numerator in~\eqref{eq:ratios}
derives from
the characteristic
\begin{eqnarray*}
\phi(t)
& = &
\1_{[0,\infty)}(t) \sum_{|v|=1} e^{-\beta(S(v)-t)} \1_{\{S(v)>t+c\}}
\\
& \leq&
\1_{[0,\infty)}(t) \sum_{|v|=1} e^{-\beta(S(v)-t)} \1_{\{S(v)>t\}}
\end{eqnarray*}
and the denominator from
\[
\psi(t) = \1_{[0,\infty)}(t) \sum_{|v|=1} e^{-\alpha(S(v)-t)} \1_{\{
S(v)>t\}}.
\]
Both $e^{-\beta t} \phi(t)$ and $e^{-\alpha t} \psi(t)$ are
decreasing in $t\geq0$. Thus, $\phi$ and $\psi$ have paths in the
Skorohod $D$-space and $\E\phi(t)$ and $\E\psi(t)$ are continuous
almost everywhere w.r.t.\ Lebesgue measure.
Thus the conditions of this form needed in Theorems 3.1 and 6.3 in
\cite{Ner1981} do hold.

Now we prove part (a) of the proposition,
where $\beta=\alpha$ in $\phi$. To this end, assume that \eqref
{eq:A4a} holds. Then $\phi$ and $\psi$ satisfy condition (3.2) of
Theorem~3.1 in~\cite{Ner1981} because
\begin{eqnarray*}
e^{-\alpha t} \phi(t)
& \leq&
e^{-\alpha t} \psi(t) \\
& = &
e^{-\alpha t} \1_{[0,\infty)}(t) \sum_{|v|=1} e^{-\alpha(S(v)-t)} \1
_{\{S(v)>t\}}
\leq \sum_{|v|=1} e^{-\alpha S(v)}
\end{eqnarray*}
for all $t \geq0$. Moreover,
\begin{eqnarray*}
\int_0^{\infty} e^{-\alpha t} \E\phi(t)\, \mathrm{d}t
& \leq&
\int_0^{\infty} e^{-\alpha t} \E\psi(t) \, \mathrm{d}t
=
\int_0^{\infty} \E\sum_{|v|=1} e^{-\alpha S(v)} \1_{\{S(v)>t\}} \,
\mathrm{d}t \\
& = &
\E\sum_{|v|=1} S(v) e^{-\alpha S(v)}
=
-m'(\alpha),
\end{eqnarray*}
where we have used Fubini's theorem. Furthermore, $-m'(\alpha)$ is
positive and finite.
Since $e^{-\alpha t} \psi(t)$ is decreasing, the integral criterion
ensures the validity of condition (3.1) of Theorem 3.1 in \cite
{Ner1981} for both $\phi$ and $\psi$.
Hence, by Theorem 3.1 of~\cite{Ner1981} and another use of \eqref
{eq:embedded_connection}, we get
\begin{eqnarray*}
e^{-\alpha t} \sum_{v \in\mathcal{T}_t} e^{-\alpha(S(v)-t)} \1_{\{
S(v)-t > c\}}
& \to&
W^{(\alpha)} \frac{\int_0^{\infty} e^{-\alpha s} \E\phi(s) \,
\mathrm{d}s}{\E S_1} \\
& = &
W^{(\alpha)} \frac{\int_c^{\infty} \Prob(S_1 > s) \, \mathrm
{d}s}{-m'(\alpha)}
\end{eqnarray*}
in probability as $t \to\infty$. For the denominator, Proposition
\ref{Prop:martingale} shows that
\[
e^{-\alpha t} \sum_{v \in\mathcal{T}_t} e^{-\alpha(S(v)-t)}
= W^{(\alpha)}_{\mathcal{T}_t}
\to W^{(\alpha)}\qquad  \mbox{a.s.}
\]
Thus, the ratio tends to $\varepsilon(c) := (-m'(\alpha))^{-1} \int
_c^{\infty} \Prob(S_1 > s) \, \mathrm{d}s$ in probability on the set
of survival $S$ as $t \to\infty$. Finally, integrability of $S_1$
ensures that~$\varepsilon(c)$ can be made arbitrarily small.

Turning to the proof of part (b), suppose that~\eqref{eq:A4b} holds,
which gives
\[
\E\sum_{|v|=1} e^{-\theta S(v)}
= m(\theta) < \infty.
\]
This implies the validity of Nerman's Condition 6.1. As for his
Condition~6.2, fix $\beta\geq\theta$, and observe that $e^{-\beta
(S(v) - t)} \leq e^{-\theta(S(v)-t)}$ on $\{S(v) > t\}$. Thus,
\[
e^{-\theta t} \1_{[0,\infty)}(t) \sum_{|v|=1} e^{-\beta(S(v)-t)} \1
_{\{S(v)>t\}}
\leq
\sum_{|v|=1} e^{-\theta S(v)},
\]
which is integrable by~\eqref{eq:A4b}. Therefore, $\phi$ and $\psi$
satisfy Nerman's Condition~6.2. Hence, by Theorem 6.3 in \cite
{Ner1981}, we infer that the ratio in the proposition tends to
$\varepsilon(c)$ a.s.\ on $S$ where $\varepsilon(c)$ is defined as in
the proof of part (a). By the same reasoning as above, $\varepsilon
(c)$ tends to $0$ as $c$ tends to $\infty$ which completes our argument.
\end{pf}

Proposition~\ref{Prop:ratios} is an essential ingredient to the proof
of the next result, which is in the spirit of Theorem 8.6 in \cite
{BK1997} and is designed to give conditions which allow \eqref
{eq:GBP_limit} to be deduced from~\eqref{eq:given}.

\begin{Theorem} \label{Thm:GBP_limit}
Suppose that~\eqref{eq:A1}--\eqref{eq:A3},~\eqref{eq:A4b} and the
following three conditions hold for a sequence $t_{n}\uparrow\infty$,
which in the $r$-geometric case takes values in $d\Z$ ($d := \log r$)
only for all $n \geq1$:
\begin{longlist}[(iii)]
\item[(i)]
There are a nonnegative function $H$ and a random variable $W$ such that
\begin{equation} \label{eq:given}
\sum_{v \in\mathcal{T}_{t_n}} e^{-\alpha S(v)} H(S(v)) \to W
\qquad \mbox{a.s.\ as } n \to\infty.
\end{equation}
\item[(ii)]
For some $h < \infty$,
\[
\varepsilon_n(a) = \biggl(\frac{H(a+t_n)}{H(t_n)} - h \biggr) \to 0
\qquad \mbox{as } n \to\infty
\]
uniformly in $a$ on compact subsets of $[0,\infty)$.
\item[(iii)]
For a finite $K$, all $a \geq0$ and all sufficiently large $n\ge1$
\begin{equation} \label{eq:outH}
\frac{H(a+t_n)}{H(t_n)} \leq K e^{(\alpha-\theta) a}.
\end{equation}
\end{longlist}
Then
\begin{equation} \label{eq:GBP_limit}
H(t_n) \sum_{v \in\mathcal{T}_{t_n}} e^{-\alpha S(v)} \to h W
\qquad \mbox{a.s. }  (n \to\infty),
\end{equation}
where in the $r$-geometric case it suffices that \textup{(ii)} holds for $a \in
d\Z$ only and uniform convergence on compact subsets of $[0,\infty)$
can be dropped.
\end{Theorem}

\begin{pf}
Note first that, by increasing $K$ if necessary, condition (iii)
implies that for all large $n$
\[
|\varepsilon_n(a)| \leq Ke^{(\alpha-\theta) a} \qquad (a \geq0).
\]
Clearly, the limits in~\eqref{eq:given} and~\eqref{eq:GBP_limit} are
both zero when $\mathcal{T}_{t_n}$ is eventually empty, and so
attention can center on the survival set $S$. For this proof let $\sum
$ be the sum over $v \in\mathcal{T}_{t_n}$. Then, considering the
ratio of the terms on the left-hand sides of~\eqref{eq:given} and
\eqref{eq:GBP_limit},
\begin{eqnarray*}
\frac{\sum e^{-\alpha S(v)} H(S(v))}{H(t_n) \sum e^{-\alpha S(v)}}
& = &
\frac{\sum e^{-\alpha S(v)} H(S(v))/H(t_n)}{\sum e^{-\alpha S(v)}} \\
& = &
\frac{\sum e^{-\alpha S(v)} (h + \varepsilon_n(S(v)-t_n))}{\sum
e^{-\alpha S(v)}} \\
& = &
h + \frac{\sum e^{-\alpha(S(v)-t_n)} \varepsilon_n(S(v)-t_n)}{\sum
e^{-\alpha(S(v)-t_n)}}.
\end{eqnarray*}
Fix $c > 0$ and note that $\delta_n := \sup\{|\varepsilon_n(a)|\dvtx  0
\leq a \leq c\}$ tends to $0$ by condition~(ii). Then
\[
\biggl|\frac{\sum e^{-\alpha(S(v)-t_n)} \varepsilon_n(S(v)-t_n)}{\sum
e^{-\alpha(S(v)-t_n)}} \biggr|
\leq
\delta_n + \frac{\sum e^{-\theta(S(v)-t_n)} K \1_{\{S(v)-t_n > c\}
}}{\sum e^{-\alpha(S(v)-t_n)}}.\vadjust{\goodbreak}
\]
Using Proposition~\ref{Prop:ratios}, the right-hand side goes to zero
as $n$ and then $c$ tends to infinity. In the $r$-geometric case the
same argument works with $\delta_n := \max\{|\varepsilon_n(a)|\dvtx  a
\in[0,c]\cap d\Z\}$, which converges to zero when the convergence in
(ii) holds for $a \in d\Z$.
\end{pf}

\section{\texorpdfstring{Proof of Theorem \protect\ref{Thm:essential_uniqueness_and_reg_var}}
{Proof of Theorem 3.1}}
\label{sec:unique}

\begin{Lemma} \label{Lem:end_reg_var_at_0}
Assume that~\eqref{eq:A1}--\eqref{eq:A3} and~\eqref{eq:E} hold and
that\break $\limsup_{n \to\infty} S_n = \infty$ a.s.
Then $D(t) = (1-\varphi(t))/t$ is slowly varying at $0$.
\end{Lemma}

Note that in the situation of the lemma, condition~\eqref{eq:A4} is
sufficient for $\limsup_{n \to\infty} S_n = \infty$ a.s.\ to hold.
The following proof is based on the proofs of Theorem 1.4 in \cite
{BK1997} and Theorem 1 in~\cite{Kyp1998}.

\begin{pf*}{Proof of Lemma~\ref{Lem:end_reg_var_at_0}}
For fixed $t > 0$, $u^{-1}(1-\varphi(ut))$ is the Laplace transform of
a $\sigma$-finite measure on $[0,\infty)$ (see, e.g.,~\cite{Fel1971}, Section
XIII.2, equation\ (2.7)) and thus so is $u^{-1}(1-\varphi
(ut))/(1-\varphi(t))$. The latter is bounded by $u^{-1} \vee1$ for $u
> 0$. A~standard selection argument shows that any sequence decreasing
to 0 contains a subsequence $(t_n)_{n \geq1}$ such that
\[
u^{-1} \frac{1-\varphi(ut_n)}{1-\varphi(t_n)}
\mathop{\longrightarrow}_{n \to\infty} l(u)\qquad  (u > 0)
\]
for some decreasing and convex function $l\dvtx (0,\infty) \to(0,\infty)$.
Now fix any such $(t_n)_{n \geq1}$ with corresponding limiting
function $l$. Then, by reproducing the following telescoping sum from
\cite{BK1997}, page 345, which is obtained from the fact that $\varphi
$ satisfies the functional equation~\eqref{eq:FE1} with $T_i^{\alpha
}$ instead of $T_i$, we get
\begin{eqnarray*} \label{eq:1st_BK's_trick}
l(u)
& = &
\lim_{n \to\infty}
\frac{1-\varphi(ut_n)}{u (1-\varphi(t_n))} \\
& = &
\lim_{n \to\infty}
\E\sum_{i\geq1} T_i^{\alpha} \frac{1-\varphi(u T_i^{\alpha}
t_n)}{uT_i^{\alpha}(1-\varphi(t_n))}
\prod_{k<i}\varphi(u t_n T_k^{\alpha}) \\
& \geq&
\E\sum_{i\geq1}\liminf_{n \to\infty} T_i^{\alpha} \frac
{1-\varphi(u T_i^{\alpha} t_n)}{uT_i^{\alpha}(1-\varphi(t_n))}
\prod_{k<i} \varphi(u t_n T_k^{\alpha}) \\
& = &
\E\sum_{i\geq1} T_i^{\alpha} l(uT_i^{\alpha})
=
\E l(ue^{-\alpha S_1}),
\end{eqnarray*}
where the inequality follows from a double application of Fatou's Lemma
and the last equality stems from~\eqref{eq:connection} with $n=1$.
Thus $(l(ue^{-\alpha S_n}))_{n \geq0}$ is a nonnegative
supermartingale and a.s.\ convergent to some finite limiting variable
$g(u)$. Here,
\[
g(u) = \lim_{n \to\infty} l(ue^{-\alpha S_n})
= \limsup_{n \to\infty} l(ue^{-\alpha S_n}) = l(0+),
\]
using the assumption that $\limsup_{n \to\infty} S_n = \infty$ a.s.
In particular, since the expectation of a supermartingale is
decreasing, $l(1) \geq\E g(1) = l(0+)$. On the other hand, by the
monotonicity of $l$, for any $0 < u \leq1$, \mbox{$l(0+) \geq l(u) \geq l(1)
= 1$}. Thus $l(u)=1$ for all $u \in(0,1]$. Since this limit is
independent of the choice of $(t_n)_{n \geq1}$, $D$ is slowly varying
at $0$.
\end{pf*}

\begin{Theorem} \label{Thm:appr_W}
Suppose that~\eqref{eq:A1}--\eqref{eq:A3},
\eqref{eq:E} and either $\E W^{(\alpha)} = 1$ or~\eqref{eq:A4b}
hold. Let $\Phi$ be the disintegration of $\varphi$ (w.r.t
$T^{(\alpha)}$).
Then
\[
\lim_{t \to\infty} e^{\alpha t}\bigl(1-\varphi(e^{-\alpha t})\bigr) \sum_{v
\in\mathcal{T}_t} L(v)^{\alpha}
= -\log\Phi(1) \qquad \mbox{a.s.}
\]
\end{Theorem}

The theorem also holds under~\eqref{eq:A1}--\eqref{eq:E}, for, from
\cite{Lyo1997},
$\E W^{(\alpha)} = 1$ is slightly weaker than~\eqref{eq:A4a}.
\begin{pf*}{Proof of Theorem~\ref{Thm:appr_W}}
Let $\newW:= -\log\Phi(1)$.
We first consider the case that $\E W^{(\alpha)} = 1$. Then
\begin{equation} \label{eq:W^alpha_along_T_t}
W^{(\alpha)}_{\mathcal{T}_t} = \sum_{v \in\mathcal{T}_t}
L(v)^{\alpha} \to W^{(\alpha)}
\qquad \mbox{a.s.\ as } t \to\infty.
\end{equation}
Now, for a contradiction, suppose that $(1-\varphi(t))/t \rightarrow
\infty$ as $t \downarrow0$. Then, for any $K > 0$, using Lemma \ref
{Lem:disintegration_via_HSL}(c) and~\eqref{eq:W^alpha_along_T_t},
\[
\newW
= \lim_{t \to\infty} \sum_{v \in\mathcal{T}_t} \bigl(1- \varphi
(L(v)^{\alpha})\bigr)
\geq \lim_{t \to\infty} \sum_{v \in\mathcal{T}_t} K L(v)^{\alpha}
= K W^{(\alpha)}\qquad \mbox{a.s.}
\]
Letting $K \uparrow\infty$ yields $\Prob(\newW= \infty) \geq\Prob
(W^{(\alpha)} > 0) > 0$.
Then $\newW= \infty$ on $S$ and is zero otherwise. Thus
$\varphi(1)=\E e^{-\newW}=1-\Prob(S)$ which contradicts the
assumptions in~\eqref{eq:E} since $\varphi(t) \downarrow1-\Prob(S)$
as $t \uparrow\infty$.
Thus, $(1-\varphi(t))/t \rightarrow c < \infty$ and so, using Lemma
\ref{dissecting}(a),
\[
\newW
= \lim_{t \to\infty} \sum_{v \in\mathcal{T}_t} \frac{1- \varphi
(L(v)^{\alpha})}{L(v)^{\alpha}}L(v)^{\alpha}
= c W^{(\alpha)}\qquad
\mbox{a.s.}
\]
which combines with~\eqref{eq:W^alpha_along_T_t} to give the result.

Now suppose that~\eqref{eq:A4b} holds.
Recall that $D(x) := x^{-1}(1-\varphi(x))$ and is slowly varying at
the origin by Lemma~\ref{Lem:end_reg_var_at_0}. Slow variation implies
that $|D(xy)/D(y) - 1| \to0$ as $y \downarrow0$ uniformly in $x$ on
compact subsets of $(0,\infty)$, and for any $\varepsilon> 0$ there
exists a finite $K$ and a $C > 0$ such that $D(xy)/D(y) \leq
Kx^{-\varepsilon}$ for all $x \leq1$ and $y \leq C$. (These
statements follow from Theorem 1.2.1 in~\cite{BGT1989}, and from the
integral representation of slowly varying functions given in Theorem
1.3.1 in~\cite{BGT1989}.) Let $H(t) := D(e^{-\alpha t})$. Then
assumptions~(ii) and~(iii) of Theorem~\ref{Thm:GBP_limit} hold, and
(i) follows from a calculation similar to that at the beginning of the
proof of Lemma~\ref{Lem:Identify_end_FP_and_disintegration}.
Therefore, Theorem~\ref{Thm:GBP_limit} completes the proof.\vadjust{\goodbreak}
\end{pf*}

\begin{pf*}{Proof of Theorem
\ref{Thm:essential_uniqueness_and_reg_var}}
Slow variation is given in Lemma~\ref{Lem:end_reg_var_at_0}. It
remains to show uniqueness up to a scale factor. Recall that
$\varphi_{\alpha}$ is the Laplace transform of $W^{(\alpha)}$.
If $\E W^{(\alpha)} = 1$, the result follows already from the proof of
Theorem~\ref{Thm:appr_W}, where we showed that
$\varphi(t)=\varphi_{\alpha}(ct)$ for some $c \in(0,\infty)$.
In the general case,
let $\widehat{W}$ be another variable, but with Laplace transform~$\widehat{\varphi}$,
satisfying~\eqref{eq:E}, and let $\widehat
{D}(t) := t^{-1}(1-\widehat{\varphi}(t))$, $t>0$. Then, by Theorem~\ref{Thm:appr_W},
\[
\lim_{t \to\infty} D(e^{-\alpha t}) \sum_{v \in\mathcal{T}_t}
L(v)^{\alpha}
= -\log\Phi(1)\qquad \mbox{a.s.}
\]
An analogous result holds for $\widehat{\varphi}$ and its
disintegration $\widehat{\Phi}$. On the other hand,
\[
\lim_{t \to\infty} \frac{D(e^{-\alpha t}) \sum_{v \in\mathcal
{T}_t} L(v)^{\alpha}}{\widehat{D}(e^{-\alpha t}) \sum_{v \in
\mathcal{T}_t} L(v)^{\alpha}}
= \lim_{t \to\infty} \frac{D(e^{-\alpha t})}{\widehat
{D}(e^{-\alpha t})} \qquad \mbox{a.s.\ on }S,
\]
that is, the limit of the ratios is a deterministic nonnegative
constant $c \in[0,\infty]$, say. This implies $-\log\Phi(1) =
c(-\log\widehat{\Phi}(1))$ a.s. Now, by Lemma \ref
{Lem:Identify_end_FP_and_disintegration}, $-\log\Phi(1)$ and $-\log
\widehat{\Phi}(1)$ are both a.s.\ finite and positive with positive
probability which implies $c \in(0,\infty)$.
Thus $\varphi(t)=\widehat{\varphi}(ct)$ in view of Lemma \ref
{Lem:disintegration varphi}.
\end{pf*}

\section{Regular variation at $0$ of fixed points} \label{sec:Reg_var_in_0}

The key to the proof of Theorem~\ref{Thm:Disintegration} is the
verification that, for any $f \in\SolM$, if~\eqref{eq:A4} holds,
$1-f$ is regularly varying at $0$ with index $\alpha$ in the
continuous case, and it is ``nearly'' regularly varying otherwise.

\begin{Theorem} \label{Thm:Reg_var_in_0}
Assuming~\eqref{eq:A1}--\eqref{eq:E}, any $f \in\SolM$ satisfies
\begin{equation} \label{eq:Reg_var_in_0}
\lim_{t \downarrow0} \frac{1-f(ut)}{1-f(t)} = u^{\alpha}
\end{equation}
for all $u \in(0,\infty)$ in the continuous case and all $u \in r^{\Z
}$ in the $r$-geometric case, where in the latter case the limit $t
\downarrow0$ is restricted to some arbitrary fixed residue class $s
r^{\Z}$, $s \in[1,r)$.
\end{Theorem}

The rest of this section is devoted to the proof of this theorem, which
is divided into five steps:
The first one provides the justification that we can assume that $T_i <
1$ a.s. for all $i \geq1$.
The second step is a standard selection argument that guarantees that,
for any solution $f \in\SolM$ and any sequence $t \downarrow0$, the
ratio $(1-f(st))/(1-f(t))$ as a~function of $s \in[0,1]$ has a~convergent subsequence.
In the third step we introduce $\SolM^{\beta}$, a~subset of the set
of fixed points containing only fixed points which show a sufficiently
regular behavior at $0$. For $f \in\SolM^{\beta}$, where $\beta:=
\theta$ if~\eqref{eq:A4b} holds and $\beta= \alpha$ otherwise, we
then prove that any limiting function, as obtained in Step 2, satisfies
a Choquet--Deny-type equation. An appeal to the theory of these
functional equations as presented in~\cite{RS1994} provides us with
a~good description of the\vadjust{\goodbreak} behavior of $f$ at $0$. The idea of utilizing
a~Choquet--Deny-type equation has been taken from the proof of Theorem~2.12 in~\cite{DL1983}.
Step~4 proves Theorem~\ref{Thm:Reg_var_in_0} under the additional
assumption that $f \in\SolM^{\beta}$.
Finally, in Step~5, we show that $\SolM^{\beta} = \SolM$.

\subsection*{Step 1: Reduction to the case $T_i < 1$ a.s.\ for all
$i \geq1$}

As in~\cite{BK2005}, Section~3, one element in the approach here is
the reduction to the simpler case when the weights $T_i$ are bounded
from above by $1$.
First, by Lemma~\ref{Lem:HSL_FPE}, $f \in\SolM$ entails that $f$
also solves~\eqref{eq:FE1} with $T$ replaced by $T^{\st}$. By
construction, $T_i^{\st} < 1$ a.s. for all $i \geq1$.
Second, Lemma~\ref{Lem:wlog_T_i<1} ensures that the validity of \eqref
{eq:A1}--\eqref{eq:A3} for $T$ carries over to $T^{\st}$ with the
same characteristic exponent $\alpha$, and the same inheritance holds
true for~\eqref{eq:A4a},~\eqref{eq:A4b} and the minimal closed
subgroup $\G(T)$, respectively.
In other words, the sequence~$T^{\st}$ also satisfies the assumptions
of Theorem~\ref{Thm:Reg_var_in_0} and also the parameters describing
the behavior of $f$ in equation~\eqref{eq:Reg_var_in_0}, the
characteristic exponent $\alpha$ and the multiplicative $\G(T)$,
coincide with the corresponding parameters for the sequence $T^{\st}$.
Consequently, it constitutes no loss of generality to prove Theorem
\ref{Thm:Reg_var_in_0} under the additional assumption
[besides~\eqref{eq:A1}--\eqref{eq:E}]
{\renewcommand{\theequation}{A\arabic{equation}}
\setcounter{equation}{5}
\begin{equation} \label{eq:A5}
T_i < 1 \qquad \mbox{a.s.\ for all } i \geq1.
\end{equation}
}
\vspace*{-\baselineskip}
\setcounter{equation}{1}

\subsection*{Step 2: The selection argument}

\begin{Lemma} \label{Lem:Selection}
Suppose that~\eqref{eq:A1}--\eqref{eq:A5} hold, and let $f \in\SolM
$. Then any sequence decreasing to zero contains a subsequence
$(t_n)_{n \geq1}$ such that, for an increasing function $g\dvtx (0,1] \to
[0,1]$ satisfying $g(1)=1$,
\begin{equation}\label{eq:selection}
\frac{1-f(ut_n)}{1-f(t_n)} \mathop{\longrightarrow}_{n \to\infty} g(u)
\end{equation}
for all $u \in(0,1]$.
\end{Lemma}
\begin{pf}
It follows from the proof of Lemma 6.2 in~\cite{AM2009} that $1-f(t) >
0$ for all $t>0$. Thus, the ratio in~\eqref{eq:selection} is well
defined. Now starting with an initial sequence decreasing to zero, we
choose a subsequence giving convergence for each rational $u \in
(0,1]$. This is possible since $(1-f(ut))/(1-f(t)) \in[0,1]$ by the
monotonicity of $f$. This defines an increasing limit, which can have
only countably many discontinuities. Now select further subsequences to
get convergence at any discontinuity and define the resulting limit to
be $g$. Obviously $g(1)=1$.
\end{pf}

\subsection*{Step 3: An application of the theory of Choquet--Deny
equations}

We introduce a subset of $\SolM$ with members that behave more
regularly at $0$. Recall that, for $f \in\SolM$, $D_{\beta}(t)$ is
$(1-f(t))/t^{\beta}$. With this notation,
\begin{equation} \label{eq:Finf^theta}
\SolM^{\beta} :=
\Bigl\{f \in\SolM\dvtx  \sup_{u \leq1, t \leq c} D_{\beta}(ut)/D_{\beta}(t)
< \infty\mbox{ for some } c > 0\Bigr\}.\hspace*{-35pt}
\end{equation}
For the rest of this section let $\beta:= \theta$ if~\eqref{eq:A4b}
holds and $\beta:= \alpha$, otherwise.\vadjust{\goodbreak}

\begin{Lemma} \label{Lem:BK's_trick}
Assume~\eqref{eq:A1}--\eqref{eq:A5} and let $f \in\SolM^{\beta}$.
Then, for any sequence decreasing to zero, there exist a subsequence
$(t_n)_{n \geq1}$ and a function~$h$ satisfying
\[
\lim_{n \to\infty} \frac{1-f(u t_n)}{1-f(t_n)} = h(u) u^{\alpha}
\]
for all $u \in(0,1]$.
In the continuous case, $h$ is one,
while in the lattice case, $h$ is strictly positive and
multiplicatively $r$-periodic with $h(1)=1$.
\end{Lemma}
\begin{pf}
For any given sequence decreasing to zero choose a subsequence
according to Lemma~\ref{Lem:Selection}, that is, a subsequence
$(t_n)_{n \geq1}$ such that the fraction $(1-f(ut_n))/(1-f(t_n))$
converges to $g(u)$ for some increasing function $g\dvtx (0,1] \to[0,1]$
satisfying $g(1)=1$. Then, as in the proof of Lemma~\ref{Lem:end_reg_var_at_0},
\begin{equation} \label{eq:BK's_trick}
\frac{1-f(ut_n)}{u^{\alpha} (1-f(t_n))}
= \E\sum_{i \geq1} T_i^{\alpha} \frac{1-f(u T_i
t_n)}{(uT_i)^{\alpha}(1-f(t_n))}
\prod_{k<i} f(u t_n T_k).
\end{equation}
Since $f \in\SolM^{\beta}$, we have
\[
T_i^{\alpha} \frac{1-f(u T_i t_n)}{(uT_i)^{\alpha}(1-f(t_n))}
\leq K T_i^{\alpha} (uT_i)^{\beta- \alpha}
= K u^{\beta- \alpha} T_i^{\beta}
\]
for sufficiently large $n$, some deterministic constant $K < \infty$
and all $i$. By the definition of $\beta$, $m(\beta)$ is finite and
thus the dominated convergence theorem yields upon letting $n \to
\infty$ in~\eqref{eq:BK's_trick}
\[
g(u)/u^{\alpha} = \E\sum_{i \geq1} T_i^{\alpha} \frac
{g(uT_i)}{(uT_i)^{\alpha}}\qquad (u \in(0,1]).
\]
Equivalently [see~\eqref{eq:connection}], $\widetilde{g}(x) :=
e^{\alpha x} g(e^{-x})$ ($x \geq0$) satisfies the following
Choquet--Deny-type functional equation:
\begin{equation} \label{eq:CDE}
\widetilde{g}(x) = \E\widetilde{g}(x + S_1)\qquad  (x \geq0).
\end{equation}
Since $g$ is increasing and bounded, $\widetilde{g}$ is locally
bounded on $[0,\infty)$ and thus locally integrable w.r.t.\ Lebesgue
measure. Moreover, since $1 = \widetilde{g}(0) = \E\widetilde
{g}(S_1)$, we obtain that $\Prob(\widetilde{g}(S_1) \geq1) > 0$,
which immediately implies that $\widetilde{g}(x_0) \geq1$ for some
\mbox{$x_0 > 0$}. This in combination with $\widetilde{g}$ being the product
of a decreasing function and a positive increasing function gives
$\widetilde{g} > 0$ on $[0,x_0]$.

Now assume first that we are in the continuous case.
Then an application of Theorem 2.2.2 in~\cite{RS1994} shows that
$\widetilde{g}$ equals a constant $c$ almost everywhere w.r.t.\
Lebesgue measure. Utilizing $\widetilde{g} > 0$ on $[0,x_0]$ yields $c
> 0$. Rewriting this in terms of $g$ gives $g(u) = cu^{\alpha}$ almost
everywhere w.r.t.\ Lebesgue measure. From this we conclude that $g(u) =
c u^{\alpha}$ for all $u \in(0,1)$ since $g$ is known to be
increasing. Furthermore, $g(1)=1$ implies $c \leq1$, but to establish
that $c=1$ needs additional reasoning.
Applying this argument a second time, for fixed $s \in(0,1)$, the
sequence $(s^{-1} t_{n})_{n \geq1}$ has a\vadjust{\goodbreak} subsequence $(s^{-1}
t_{n}')_{n \geq1}$ such that for some $c' \in(0,1]$
\[
\lim_{n \to\infty}
\frac{1-f(u s^{-1} t_{n}')}{1-f(s^{-1} t_{n}')} = c' u^{\alpha}
\]
holds for all $u \in(0,1)$.
It now constitutes no loss of generality to assume that $(t_{n})_{n
\geq1} = (t_{n}')_{n \geq1}$.
Then
\begin{eqnarray*}
c' (us)^{\alpha}
&=&
\lim_{n \to\infty} \frac{1-f((us) s^{-1} t_{n})}{1-f(s^{-1}
t_{n})}\\
&=&
\lim_{n \to\infty} \frac{1-f(u t_{n})}{1-f(t_{n})}
\frac{1-f(s s^{-1} t_{n})}{1-f(s^{-1} t_{n})}
 = c u^{\alpha} c' s^{\alpha} = cc' (us)^{\alpha}.
\end{eqnarray*}
Since $c' > 0$ this implies that $c = 1$.

In the lattice case we have that $S_1$ is confined to $\Z^d$ with $d
:= \log r$. Then Corollary 2.2.3 in~\cite{RS1994} yields $\widetilde
{g}(x+nd)=\widetilde{g}(x)$ for all $x \geq0$ and $n \in\N_0$; that
is, $\widetilde{g}$ is $d$-periodic. This immediately provides us with
the identity $g(u) = \widetilde{g}(-\log u) u^{\alpha} =: h(u)
u^{\alpha}$ ($u \in(0,1]$) where $h(u) = \widetilde{g}(-\log u)$ is
multiplicatively $r$-periodic. The fact that $h$ is strictly positive
follows from the monotonicity of $g$ in combination with $g(1)=1$ and
the periodicity of $h$.
\end{pf}

\subsection*{\texorpdfstring{Step 4: Proof of Theorem \protect\ref{Thm:Reg_var_in_0} for $f \in\SolM^{\beta}$}
{Step 4: Proof of Theorem 11.1 for $f \in\SolM^{\beta}$}}

Let $f \in\SolM^{\beta}$. It suffices to show that for any sequence
$t_n \downarrow0$ (where $t_n$ is chosen from a fixed residue class of
$\R^{+} \operatorname{mod} r^{\Z}$ in the $r$-geometric case) there exists a
subsequence such that the convergence in~\eqref{eq:Reg_var_in_0} holds
along this subsequence on $\G(T) \cap(0,1]$ ($\G(T)$ is the closed
multiplicative subgroup generated by $T$).
This is what Lemma~\ref{Lem:BK's_trick} does.

\subsection*{Step 5: Proof that $\SolM^{\beta} = \SolM$}

In the fifth step, we fix $f \in\SolM$ with disintegration $M$ and
show that $D_{\beta}(t) = t^{-\beta}(1-f(t))$ satisfies the growth
condition in the definition of the set $\SolM^{\beta}$ in equation\
\eqref{eq:Finf^theta} and, thus, that $f \in\SolM^{\beta}$. To this
end, let $\newW:= -\log M(1)$. Then, by Lemma~\ref
{Lem:disintegration_via_HSL}(c),
\begin{equation} \label{eq:Fbar_disintegration_via_HSL}
\lim_{t \to\infty} \sum_{v \in\mathcal{T}_t} e^{-\alpha S(v)}
D_{\alpha}\bigl(e^{-S(v)}\bigr) = \newW\qquad \mbox{a.s.}
\end{equation}
As in Lemma~\ref{Lem:disintegration varphi} and Theorem \ref
{Thm:appr_W}, let $\Phi$ be the disintegration of $\varphi$, and
recall that $D(t) = t^{-1}(1-\varphi(t))$, which is slowly varying at
$0$. Applying Lemma~\ref{Lem:disintegration_via_HSL}(c) again,
\begin{equation} \label{eq:W}
\lim_{t \to\infty} \sum_{v \in\mathcal{T}_t} e^{-\alpha S(v)}
D\bigl(e^{-\alpha S(v)}\bigr)= W
\end{equation}
with $W := -\log\Phi(1)$, where $W$ has Laplace transform $\varphi$
and is an endogenous fixed point w.r.t.\ $T^{(\alpha)}$ (see Lemmas
\ref{Lem:Identify_end_FP_and_disintegration}(a) and \ref
{Lem:disintegration varphi}).

The idea now is to bound $D_{\alpha}$ using $D$ and thereby to bound
the behavior of $D_{\alpha}$ at zero. Let
\[
\Kl := \liminf_{t \to\infty} \frac{D_{\alpha
}(e^{-t})}{D(e^{-\alpha t})}
\quad \mbox{and}\quad
\Ku := \limsup_{t \to\infty} \frac{D_{\alpha
}(e^{-t})}{D(e^{-\alpha t})}.
\]

The next lemma gives the only property of $D_{\alpha}$ in addition to
\eqref{eq:Fbar_disintegration_via_HSL} that is relevant for the
subsequent results in the fifth step.

\begin{Lemma} \label{Lem:mon_bound}
For any $c > 0$ there is a $\delta> 0$ such that
\[
\frac{D_{\alpha}(e^{-(x+a)})}{D_{\alpha}(e^{-x})} \leq e^{\delta}
\quad \mbox{and}\quad
\frac{D_{\alpha}(e^{-(x-a)})}{D_{\alpha}(e^{-x})} \geq e^{-\delta}
\]
for all $x \in\R$ and $0 \leq a \leq c$.
\end{Lemma}
\begin{pf}
Recall that $1-f(t) > 0$ for all $t > 0$ by~\cite{AM2008}, Lemma 6.2.
Since $e^{-\alpha x} D_{\alpha}(e^{-x}) = 1-f(e^{-x})$ decreases,
\[
\frac{D_{\alpha}(e^{-(x+a)})}{D_{\alpha}(e^{-x})}
=
\frac{e^{-\alpha(x+a)} D_{\alpha}(e^{-(x+a)})}{e^{-\alpha(x+a)}
D_{\alpha}(e^{-x})}
\leq
\frac{e^{-\alpha x} D_{\alpha}(e^{-x})}{e^{-\alpha(x+a)} D_{\alpha
}(e^{-x})} \\
= e^{\alpha a} \leq e^{\alpha c}
\]
for any $0 \leq a \leq c$. The second estimation is just the reciprocal
of the first.
\end{pf}

\begin{Lemma} \label{Lem:K_l_K_u}
Under~\eqref{eq:A1}--\eqref{eq:A5}, the following assertions are true:
\begin{longlist}[(a)]
\item[(a)]
$0 < \Kl\leq\Ku< \infty$;
\item[(b)]
$\varphi(\Ku t^{\alpha}) \leq f(t)
\leq\varphi(\Kl t^{\alpha})$
for all $t \geq0$;
\item[(c)]
$\Kl D(\Kl t^{\alpha}) \leq D_{\alpha}(t)
\leq\Ku D(\Ku t^{\alpha})$ for all $t \geq0$.
\end{longlist}
\end{Lemma}
\begin{pf}
Lemma~\ref{Lem:disintegration_via_HSL}(c) and Theorem \ref
{Thm:appr_W} imply that
\begin{equation} \label{twolimits}
\qquad \lim_{t \to\infty} D (e^{-\alpha t} ) \sum_{v \in\mathcal{T}_t}
e^{-\alpha S(v)}
= W
= \lim_{t \to\infty} \sum_{v \in\mathcal{T}_t} e^{-\alpha S(v)} D
\bigl(e^{-\alpha S(v)} \bigr)\qquad \mbox{a.s.}\hspace*{-5pt}
\end{equation}
Since $D$ is decreasing
\begin{eqnarray*}
&&\lim_{t \to\infty} \frac{\sum_{v \in\mathcal{T}_t} e^{-\alpha
S(v)} D(e^{-\alpha S(v)}) \1_{\{S(v) \leq t+c\}}}{\sum_{v \in\mathcal
{T}_t} e^{-\alpha S(v)} D(e^{-\alpha S(v)})} \\
&&\qquad  \geq
\lim_{t \to\infty} \frac{D(e^{-\alpha t}) \sum_{v \in\mathcal
{T}_t} e^{-\alpha S(v)} \1_{\{S(v)-t \leq c\}}}{\sum_{v \in\mathcal
{T}_t} e^{-\alpha S(v)} D(e^{-\alpha S(v)})} \\
&&\qquad  = \lim_{t \to\infty} \frac{\sum_{v \in\mathcal{T}_t}
e^{-\alpha S(v)} \1_{\{S(v)-t \leq c\}}}{\sum_{v \in\mathcal{T}_t}
e^{-\alpha S(v)}} \frac{D (e^{-\alpha t} ) \sum_{v \in\mathcal
{T}_t} e^{-\alpha S(v)}}{\sum_{v \in\mathcal{T}_t} e^{-\alpha S(v)}
D(e^{-\alpha S(v)})}.
\end{eqnarray*}
Now, by Proposition~\ref{Prop:ratios} with $\beta= \alpha$, the
first term tends to a limit $\geq1-\varepsilon$ for given
$\varepsilon> 0$ when $c$ is large enough. The second goes to one by
\eqref{twolimits} on $\{0<W<\infty\}$,\vadjust{\goodbreak} which almost surely coincides
with $S$, the survival set.
Now, using Lemma~\ref{Lem:mon_bound} and that $D(e^{-\alpha x})$ is
increasing in $x$,
\begin{eqnarray*}
\sum_{v \in\mathcal{T}_t} e^{-\alpha S(v)} D_{\alpha}\bigl(e^{-S(v)}\bigr)
& \geq&
\sum_{v \in\mathcal{T}_t} e^{-\alpha S(v)} D_{\alpha}\bigl(e^{-S(v)}\bigr) \1
_{\{S(v) \leq t+c\}} \\
& \geq&
e^{-\delta} D_{\alpha}\bigl(e^{-(t+c)}\bigr) \sum_{v \in\mathcal{T}_t}
e^{-\alpha S(v)} \1_{\{S(v) \leq t+c\}} \\
& \geq&
e^{-\delta} \frac{D_{\alpha}(e^{-(t+c)})}{D(e^{-\alpha(t+c)})}
\sum_{v \in\mathcal{T}_t} e^{-\alpha S(v)} D\bigl(e^{-\alpha S(v)}\bigr) \1_{\{
S(v) \leq t+c\}}
\end{eqnarray*}
for some $\delta> 0$. Therefore, letting $t \to\infty$ along an
appropriate sequence,
\begin{equation} \label{eq:K_u_finite}
\overline{W} \geq e^{-\delta} \Ku(1-\varepsilon) W\qquad  \mbox{a.s.}
\end{equation}
Since $\E\Phi(1) = \varphi(1) < 1$, we have $1-q := \Prob(W > 0) >
0$. On the other hand, as a consequence of the regular variation of
$1-\varphi$ at $0$, finiteness of $\Ku$ is not affected by replacing
$f(t)$ by $f(ct)$ for $c>0$, although the numerical value of $\Ku$ may
change. Thus, by rescaling $f$ in this way, we can assume that $f(1) >
q$. Then, $f(1)=\E e^{-\overline{W}} > q$ and so $\Prob(\overline
{W} = \infty) < 1-q$. Consequently, $\Prob(W > 0, \overline
{W}<\infty) > 0$. We now conclude from~\eqref{eq:K_u_finite} that
$\Ku$ is finite for the rescaled $f$ and thus also for the original
$f$. Then, using Lemma~\ref{Lem:disintegration_via_HSL}(c) and the
slow variation of $D$ at $0$, for any $t > 0$,
\begin{eqnarray*}
-\log M(t)
& = &
\lim_{u \to\infty} \sum_{v \in\mathcal{T}_u} t^{\alpha}
e^{-\alpha S(v)} D_{\alpha}\bigl(te^{-S(v)}\bigr) \\
& \leq&
\lim_{u \to\infty} \sum_{v \in\mathcal{T}_u} t^{\alpha}
e^{-\alpha S(v)} \Ku D\bigl(te^{-\alpha S(v)}\bigr) \\
& = &
t^{\alpha} \Ku W\qquad \mbox{a.s.}
\end{eqnarray*}
After an appeal to~\eqref{eq:Disintegration_integrated}, we deduce
that $f(t) \geq\varphi(\Ku t^{\alpha})$, where we used that $W$ has
Laplace transform $\varphi$. This proves the second half of each of
(a) and (b).

In a similar way, using Lemma~\ref{Lem:mon_bound} and that
$D(e^{-\alpha x})$ is increasing in $x$,
\begin{eqnarray*}
\sum_{v \in\mathcal{T}_t} e^{-\alpha S(v)} D_{\alpha}\bigl(e^{-S(v)}\bigr)
& \leq&
e^{\delta} D_{\alpha}(e^{-t})
\sum_{v \in\mathcal{T}_t} e^{-\alpha S(v)} \1_{\{S(v) \leq t+c\}} \\
& &
{}+ \sum_{v \in\mathcal{T}_t} e^{-\alpha S(v)} D_{\alpha}\bigl(e^{-S(v)}\bigr)
\1_{\{S(v) > t+c\}} \\
& \leq&
e^{\delta} \frac{D_{\alpha}(e^{-t})}{D(e^{-\alpha t})}
\sum_{v \in\mathcal{T}_t} e^{-\alpha S(v)} D\bigl(e^{-\alpha S(v)}\bigr) \1_{\{
S(v) \leq t+c\}} \\
& &
{}+ \sum_{v \in\mathcal{T}_t} e^{-\alpha S(v)} D_{\alpha}\bigl(e^{-S(v)}\bigr)
\1_{\{S(v) > t+c\}}.
\end{eqnarray*}
Letting $t$ tend to infinity along an appropriate sequence, we obtain
with the help of Proposition~\ref{Prop:ratios}
\[
\overline{W}
\leq
e^{\delta} \Kl W + \Ku\varepsilon W
=
(e^{\delta} \Kl+ \Ku\varepsilon) W\qquad
\mbox{a.s.}
\]
where $\varepsilon> 0$ depends on the choice of $c$. Since $\E M(1) =
f(1) < 1$ by~\cite{AM2008}, Lemma~6.2, we have $\Prob(\overline{W} >
0) > 0$. On the other hand, $W < \infty$ a.s.\ by~\eqref{eq:E}. Then,
since $\varepsilon$ can be made arbitrarily small, $\Kl> 0$ follows,
for otherwise $\overline{W}=0$ a.s. Now arguing as in the first part
of the proof, we obtain $f(t) \leq\varphi(\Kl t^{\alpha})$, $t > 0$.
Part (c) is just a rearrangement of part (b).
\end{pf}

\begin{Lemma} \label{Lem:Finf=Finfalpha}
Assuming~\eqref{eq:A1}--\eqref{eq:A5}, we have that $\SolM^{\beta}
= \SolM$.
\end{Lemma}
\begin{pf}
By Lemma~\ref{Lem:K_l_K_u}, $\varphi(\Ku t^{\alpha}) \leq f(t) \leq
\varphi(\Kl t^{\alpha})$ for all $t \geq0$. Thus,
\begin{equation} \label{eq:reg_ineq}
\frac{1-f(ut)}{1-f(t)} \leq \frac{1-\varphi(\Ku(ut)^{\alpha
})}{1-\varphi(\Kl t^{\alpha})}
\end{equation}
for all $u \geq0$ and $t > 0$.

Suppose first that~\eqref{eq:A4a} holds. Then we can assume w.l.o.g.\
that $W = W^{(\alpha)}$. Then $\varphi$ is differentiable at $0$ with
derivative $-1$ so that
\[
\frac{1-\varphi(\Ku(ut)^{\alpha})}{1-\varphi(\Kl t^{\alpha})}
=
\frac{1-\varphi(\Ku(ut)^{\alpha})}{\Ku(ut)^{\alpha}}
\frac{\Kl t^{\alpha}}{1-\varphi(\Kl t^{\alpha})}
\frac{\Ku}{\Kl}u^{\alpha}
\leq
C \frac{\Ku}{\Kl} u^{\alpha}
\]
for some $C < \infty$, all $u \leq1$ and all sufficiently small $t > 0$.

The situation is more delicate if~\eqref{eq:A4b} is assumed instead of
\eqref{eq:A4a}. We show that for $f \in\SolM$ and arbitrary
$\varepsilon> 0$, there exist $K,c > 0$ such that
\begin{equation} \label{eq:eps_ineq}
\frac{1-f(ut)}{u^{\alpha} (1-f(t))} \leq K u^{-\varepsilon}
\end{equation}
for all $u \leq1$ and all $t \leq c$. We deduce from \eqref
{eq:reg_ineq} that (keep in mind that $D(s) = (1-\varphi(s))/s$ is
decreasing in $s$)
\begin{eqnarray*}
\frac{1-f(ut)}{u^{\alpha} (1-f(t))}
&\leq& \frac{\Ku}{\Kl} \frac{D(\Ku t^{\alpha})}{D(\Kl t^{\alpha})}
\frac{D(\Ku(ut)^{\alpha})}{D(\Ku t^{\alpha})}\\
&\leq& \frac{\Ku}{\Kl} \frac{D(\Ku(ut)^{\alpha})}{D(\Ku t^{\alpha})}.
\end{eqnarray*}
An application of Theorem 1.3.1 in~\cite{BGT1989} to the slowly
varying function $D$ shows that the last ratio can be bounded from
above by a constant times $u^{\varepsilon}$ in a right neighborhood of
$0$; in other words, we have established~\eqref{eq:eps_ineq}. Since we
can choose $\varepsilon\leq\alpha- \theta$, the proof is complete.
\end{pf}

\section{\texorpdfstring{The proofs of Theorems \protect\ref{Thm:Disintegration} and \protect\ref{Thm:endogeny}}
{The proofs of Theorems 8.3 and 6.2}} \label{sec:Proof_disintegration}
\mbox{}
\begin{pf*}{Proof of Theorem~\ref{Thm:Disintegration}}
Let $f \in\SolM$ and let $M$ denote the corresponding disintegrated
fixed point. Then, using~\eqref{eq:Reg_var_in_0}, we obtain from
\eqref{eq:W*_st} and~\eqref{eq:u^alpha_W*_t} in the proof of Lemma
\ref{Lem:Identify_end_FP_and_disintegration} that\vadjust{\goodbreak} for any $u>0$ and $s
= 1$ (nongeometric case) or $u \in r^{\Z}$ and $s \in(r^{-1},1]$
($r$-geometric case)
\[
- \log M(su)
=
u^{\alpha} (- \log M(s)) \qquad \mbox{a.s.}
\]
Moreover, $-\log M(s)$ is an endogenous fixed point w.r.t. $T^{(\alpha
)}$ by Lemma~\ref{Lem:Identify_end_FP_and_disintegration}. Putting $W
= -\log M(1)$, we see that $M$ satisfies~\eqref{eq:Disintegration} in
the continuous case. In the $r$-geometric case, Proposition \ref
{Prop:uniqueness_of_end_FP} comes into play because it ensures that for
any $s \in(r^{-1},1]$ there exists a constant $h(s) > 0$ such that
$-\log M(s) = h(s) s^{\alpha} W$ a.s. Now we define $h(us) := h(s)$
for $u \in r^{\Z}$ and $s \in(r^{-1},1]$. Thus, $h$ is defined on the
whole positive halfline $(0,\infty)$. Using $- \log M(su) = u^{\alpha
} (- \log M(s))$ a.s.\ for $u \in r^{\Z}$ and $s \in(r^{-1},1]$, we
see that $M$ has a representation as in~\eqref{eq:Disintegration} in
the $r$-geometric case as well. To see that $h \in\mathfrak{H}_r$ it
remains to prove that $t \mapsto h(t)t^{\alpha}$ is increasing. But in
view of~\eqref{eq:Disintegration} and \eqref
{eq:Disintegration_integrated}, this immediately follows from the
monotonicity of~$f$.\looseness=1

We have shown so far that for any disintegrated fixed point $M$ there
exist an endogenous fixed point $W$ and some function $h \in\mathfrak
{H}_r$ such that~\eqref{eq:Disintegration} holds. Since endogenous
fixed points are unique up to scaling by Proposition \ref
{Prop:uniqueness_of_end_FP} and $\mathfrak{H}_r$ is invariant under
scaling with positive factors, it is clear that one can choose $W$
independent of $f$.
\end{pf*}

Before we prove Theorem~\ref{Thm:endogeny}, we need some more
terminology. First, given the sequence $T$, the smoothing transform on
the set $\mathcal{P}(\R^+)$ of probability distributions on $\R^+$
maps a distribution $P \in\mathcal{P}(\R^+)$ to the distribution of
$\sum_{i \geq1} T_i X_i$ where $(X_i)_{i \geq1}$ is a sequence of
i.i.d. random variables with common distribution $P$. The corresponding
bivariate smoothing transform maps a distribution $P \in\mathcal
{P}(\R^+\times\R^+)$ to the distribution of $(\sum_{i \geq1} T_i
X_i, \sum_{i \geq1} T_i Y_i)$ where $(X_1,Y_1),(X_2,Y_2),\ldots$ is
a sequence of i.i.d. two-dimensional random vectors with common
distribution $P$. Notice that the bivariate transform uses the same
realization of $T$ in both components.

\begin{pf*}{Proof of Theorem~\ref{Thm:endogeny}}
Let $P$ be a distribution solving the distributional recursion \eqref
{eq:sm-fe} with Laplace transform $\varphi$, and let $(M_n(t))_{n \geq
0}$ for $t\ge0$ be the corresponding multiplicative martingales. By
Theorem~\ref{Thm:Disintegration}, their limits are given by $M(t) =
\exp(-hWt)$ a.s.\ for some $h>0$ where $W$ is endogenous w.r.t.
$T^{(\alpha)}$. By Lemma~\ref{Lem:Disintegration}, $\E M(t) = \varphi
(t)$ for all $t \geq0$, and thus $hW$ has Laplace transform $\varphi$
and distribution $P$. By replacing $W$ by $hW$, we can assume w.l.o.g.
that $h=1$. Thus the definition of endogenous fixed points w.r.t.
$T^{(\alpha)}$ entails the existence of an endogenous RTP with
marginal~$P$. The form~\eqref{eq:form_of_the_RTP} of the RTP follows
from Lemma~\ref{Lem:disintegration_via_HSL}(a).
To apply Theorem 11(c) in~\cite{AB2005}, consider the bivariate
Laplace transform $\psi_n$ of the $n$fold application of the bivariate
smoothing transform to the product measure $P \otimes P$. Denote by
$(X(v))_{v \in\V}$ and $(Y(v))_{v \in\V}$ two independent families
of i.i.d. random variables with distribution $P$. Then, for $(s,t) \in
[0,\infty)^2$, we have
\begin{eqnarray*}
\psi_n(s,t)
& = &
\E\exp\biggl(-s\sum_{|v|=n}L(v) X(v) - t \sum_{|v|=n} L(v)Y(v) \biggr) \\
& = &
\E \biggl( \E \biggl[ \exp \biggl(-s\sum_{|v|=n}L(v) X(v) - t \sum_{|v|=n} L(v)Y(v)
\biggr) \Big| \A_n \biggr] \biggr) \\
& = &
\E M_n(s) M_n(t)
 \to
\E M(s)M(t)
= \E e^{-(s+t) W} \qquad \mbox{as } n \to\infty.
\end{eqnarray*}
By the continuity theorem for Laplace transforms, the associated
distribution converges weakly to $\Prob((W,W) \in\cdot)$.
Invoking Theorem 11(c) in~\cite{AB2005}, it now follows that the
endogeny property holds, which means that any RTP with marginal $P$ is
endogenous. Further, Proposition~\ref{Prop:uniqueness_of_end_FP}
ensures that the endogenous RTP with marginal $P$ is unique. Since $P$
was an arbitrary solution to~\eqref{eq:sm-fe} and since any other
solution differs only by a scale factor, assertion (a) follows.

Turning to assertion (b), let $\{W_u\dvtx u \in\V\}$ be an endogenous RTP
associated with equation~\eqref{eq:SumFP} and $\alpha< 1$. It
suffices to show that $W_{u} = 0$ a.s. for all $u$. Assume that $\Prob
(W_{\varnothing}>0) > 0$. Using endogeny and equation \eqref
{eq:form_of_the_RTP}, we~get\looseness=-1
\[
W_{u}
=
\lim_{n \to\infty} \sum_{|v|=n} 1-\varphi([L(v)]_u) \qquad \mbox{a.s.}
\]\looseness=0
On the other hand, by Theorem~\ref{Thm:Reg_var_in_0}, the
corresponding Laplace transform~$\varphi$ satisfies $(1-\varphi
(st))/(1-\varphi(t)) \to s^{\alpha}$ as $t \downarrow0$, where in
the $r$-geometric case $s \in r^{\Z}$ and the limit $t \to0$ is
restricted to $t \in r^{\Z}$. Consequently, for all $n \geq0$,
\begin{eqnarray*}
\sum_{|u|=n} L(u) W_u
= W_{\varnothing}
& = &
\lim_{k \to\infty} \sum_{|v|=n+k} 1-\varphi(L(v)) \\
& = &
\lim_{k \to\infty} \sum_{|u|=n} \sum_{|v|=k} 1-\varphi
(L(u)[L(v)]_u) \\
& \geq&
\sum_{|u|=n} L(u)^{\alpha} \lim_{k \to\infty} \sum_{|v|=k}
1-\varphi([L(v)]_u) \\
& = &
\sum_{|u|=n} L(u)^{\alpha} W_u \qquad \mbox{a.s.}
\end{eqnarray*}
But $\sup_{|u|=n} L(u) \to0$ a.s.\ [see Lemma~\ref{dissecting}(a)],
contradicting the inequality.
\end{pf*}

\section{Solutions in other sets of functions} \label{sec:non-monotonic}

The arguments characterizing monotonic solutions can be modified to
apply to other classes of functions.
What matters is how the functions in the class behave near the
origin.
A~function $f$ will be called\vadjust{\goodbreak} \emph{eventually uniformly continuous}
if it is uniformly continuous on $[K,\infty)$ for some finite $K$.
Then the new class is the set $\mathcal{U}$ consisting of all
functions $f\dvtx [0,\infty) \to[0,1]$ with $f(0)=1$ and $f(t) \to1$ as
$t \downarrow0$ such that $\log(1-f(e^{-z}))$ is eventually uniformly
continuous. [It should be possible to widen this class further, to
functions that are c\`{a}dl\`{a}g with $\log(1-f(e^{-z}))$
having a suitably behaved modulus of continuity, but that has not been
attempted.] Note that when $f \in\mathcal{U}$ it is automatic that
$f(t) < 1$ for all small enough $t > 0$.
We define $\SolU$ to be the set of functions $f \in\mathcal{U}$
solving the functional equation~\eqref{eq:FE1}. Much of the argument
carries over.
Lemma~\ref{Lem:Selection} is the first where the argument needs some
more substantial change.

\begin{Lemma} \label{Lem:evucon_selection} Lemma~\ref{Lem:Selection}
holds for
$f \in\SolU$ with $g$ continuous (rather than increasing).
\end{Lemma}
\begin{pf}
The functions
$
H_t(z) = \log(1-f(te^{-z})) - \log(1-f(t))$ $(z \geq0)
$
are equicontinuous for all small enough $t$ and uniformly bounded at
$z=0$. Hence, by the Arzela--Ascoli theorem, for any sequence
decreasing to zero, there is a subsequence $(t_n)_{n \geq1}$ and a
continuous function $h$ such that
\[
H_{t_n}(z) = \log\biggl(\frac{1-f(t_n e^{-z})}{1-f(t_n)} \biggr)
\to h(z)\qquad  (z \geq0).
\]
The asserted convergence follows with $g(u) := \exp(h(-\log u))$, $u
\in(0,1]$.
\end{pf}

Using Lemma~\ref{Lem:evucon_selection} it is readily seen that Lemma
\ref{Lem:BK's_trick} also holds for $f \in\SolU^\beta$, which has
the natural definition.
Continuity, rather than monotonicity, is used to show that the limiting
function $\widetilde{g}$ in~\eqref{eq:CDE} satisfies $\widetilde
{g}>0$ on an interval including $0$ and then continuity implies $c=1$,
without the additional argument. Uniform continuity readily yields that
when $f \in\mathcal{U}$, the conclusion of Lemma~\ref{Lem:mon_bound} holds.
In this way the following theorem is obtained.
For it let $\mathfrak{C}_r$ be positive constants when $r=1$ and
positive, continuous, multiplicatively $r$-periodic functions otherwise.

\begin{Theorem} \label{Thm:evucon_solutions}
Suppose that conditions~\eqref{eq:A1}--\eqref{eq:E} hold. Then $\SolU
$ is given by the family in~\eqref{eq:F_inf=W_Lambda_alpha} when
parametrized by $h \in\mathfrak{C}_r$.
\end{Theorem}

\section*{Acknowledgments}

We are grateful to two anonymous referees for a careful reading of the
manuscript and for directing our attention to~\cite{AB2005}, Open
Problem~18.


%

\printaddresses

\end{document}